\documentclass[a4paper,12pt, USenglish,cleveref,autoref,thm-restate]{article}

\title{Topological simplification guided by forbidden regions}

\usepackage{amssymb}
\usepackage{makecell}
\usepackage{centernot}
\usepackage[ruled,vlined,linesnumbered]{algorithm2e}
\usepackage{amsthm}
\usepackage{thm-restate}
\usepackage{mathtools}
\usepackage{float}
\usepackage[normalem]{ulem}
\usepackage{xcolor}
\usepackage{geometry}
\usepackage{enumitem}
\usepackage{caption}
\usepackage[T1]{fontenc}
\usepackage[utf8]{inputenc}
\usepackage{appendix}
\usepackage{etoolbox}
\usepackage{hyperref}
\usepackage{cleveref}  

\newtheorem{theorem}{Theorem}
\newtheorem{lemma}[theorem]{Lemma}
\newtheorem{corollary}[theorem]{Corollary}
\newtheorem{proposition}[theorem]{Proposition}
\newtheorem{definition}[theorem]{Definition}
\newtheorem{observation}[theorem]{Observation}

\DontPrintSemicolon
\SetKw{KwAnd}{and}
\SetKw{KwOr}{or}
\SetKw{KwForAll}{for all}
\SetKw{KwWhile}{while}
\SetKw{KwDo}{do}
\SetKw{KwThen}{then}
\SetKw{KwIn}{in}
\SetKw{KwEq}{=} 

\newcommand{\proofappendix}{\textit{\small Proof may be found in Appendix.}}

\newcommand{\mf}{h}
\newcommand{\omf}{h'}

\DeclareMathOperator{\bd}{BD}

\newcommand{\diagbd}{\overline{\bd}}
\newcommand{\upbd}{\hat{\bd}}

\DeclareMathOperator{\low}{low}

\AfterEndEnvironment{theorem}{\vspace{-\parskip}\vspace{0pt}}
\AfterEndEnvironment{definition}{\vspace{-\parskip}\vspace{0pt}}
\AfterEndEnvironment{lemma}{\vspace{-\parskip}\vspace{0pt}}
\AfterEndEnvironment{corollary}{\vspace{-\parskip}\vspace{0pt}}
\AfterEndEnvironment{observation}{\vspace{-\parskip}\vspace{0pt}}

\makeatletter
\def\thm@space@setup{%
  \thm@preskip=10pt
  \thm@postskip=10pt
}
\makeatother

\newcommand {\mm}[1] {\ifmmode{#1}\else{\hmBox{\(#1\)}}\fi}
\newlist{romanenumerate}{enumerate}{1}
\setlist[romanenumerate]{label=(\roman*)}
\newlist{bracketenumerate}{enumerate}{1}
\setlist[bracketenumerate]{label=(\arabic*)}

\newcommand{\dmf}{\text{dMf}}

\newcommand{\mU}{U}

\newcommand{\hmU}{\mU}
\newcommand{\hmUaft}{\hat{\mU}}
\newcommand{\hmUbef}{\mU}

\newcommand{\cmU}{\mU^{\perp}}
\newcommand{\cmUaft}{\hat{\mU}^{\perp}}
\newcommand{\cmUbef}{\mU^{\perp}}

\newcommand{\fhR}{\mathcal{R}^\urcorner}
\newcommand{\fcR}{\mathcal{R}^\llcorner}

\newcommand{\mB}{D}
\newcommand{\hmB}{D}
\newcommand{\hmBaft}{\hat{\mB}}

\newcommand{\cmB}{D^{\perp}}

\newcommand{\mR}{R}
\newcommand{\hmR}{R}
\newcommand{\cmR}{R^{\perp}}

\newcommand{\mV}{V}
\newcommand{\cmV}{V^{\perp}}

\newcommand{\mD}{D}
\newcommand{\cmD}{D^{\perp}}

\newcommand{\cx}{x}
\newcommand{\cy}{y}
\newcommand{\cz}{z}
\newcommand{\cp}{p}
\newcommand{\cq}{q}
\newcommand{\ce}{e}

\newcommand{\ca}{a}

\newcommand{\cc}{c}

\newcommand{\bth}[1]{#1^\circ}
\newcommand{\dth}[1]{#1^{\!\times}}
\newcommand{\pbeg}[1]{#1^\sqsubset}
\newcommand{\pend}[1]{#1^\sqsupset}

\newcommand{\homo}{\xrightarrow{\times}}
\newcommand{\nohomo}{\centernot{\xrightarrow{\times}}}

\newcommand{\rel}{\rightarrow}
\newcommand{\norel}{\nrightarrow}

\newcommand{\cohomo}{\xrightarrow{\circ}}
\newcommand{\nocohomo}{\centernot{\xrightarrow{\circ}}}

\newcommand{\RR}{\mathbb{R}}
\newcommand{\ZZ}{\mathbb{Z}}
\newcommand{\bds}{\alpha}
\newcommand{\bdt}{\beta}
\newcommand{\bda}{\xi}
\newcommand{\bdb}{\gamma}

\newcommand{\NN}{\mathbb{N}}

\newcommand{\cV}{\mathcal{V}}

\newcommand{\cM}{\mathcal{M}}

\newcommand{\morsehmU}{\mU_{\cM}}
\newcommand{\morsecmU}{\mU_{\cM}^{\perp}}

\newcommand{\cA}{\mathcal{A}}

\newcommand{\pathhs}{Path}

\def\pathh#1{\overset{#1}{\rightsquigarrow}}

\def\notpathh#1{\overset{#1}{\not\rightsquigarrow}}

\def\card#1{\##1}

\DeclareMathOperator{\crit}{Crit}
\DeclareMathOperator{\vect}{Vec}

\def\pathh#1{\overset{#1}{\rightsquigarrow}}

\author{Jakub Leśkiewicz, Bartosz Furmanek, Michał Lipiński, Dmitriy Morozov}

\begin{document}

\maketitle
\footnotetext{
\textit{Jakub Leśkiewicz}: The research was partially funded by the Polish National Science Center
under Opus Grant No. 2019/35/B/ST1/00874 and Opus Grant 2025/57/B/ST1/00550.
\textit{Bartosz Furmanek}: The research was partially funded by the Polish National Science Center under
Opus Grant No. 2019/35/B/ST1/00874 and Opus Grant 2025/57/B/ST1/00550.
\textit{Michał Lipiński}: This project has received funding from the European Union’s Horizon 2020 research
and innovation programme under the Marie Skłodowska-Curie Grant Agreement No. 101034413.
\textit{Dmitriy Morozov}: This work was supported in part by the U.S. Department of Energy, Office
of Science, Office of Advanced Scientific Computing Research, under Contract No. DE-AC02-
05CH11231.
}
\begin{abstract}

\noindent \emph{Topological simplification} is the process of reducing complexity of a~function while maintaining its essential features. Its goal is to find a~new filter function, which reorders cells of the input complex in a~way which eliminates some persistent homological features, without affecting the rest. 
We present a~new approach to simplification based on the concept of \emph{forbidden regions} and combinatorial dynamics. It allows us to reorder and cancel critical values, whose cancellation is not possible using existing methods because they are not consecutive in the total order. Each such cancellation takes $O(c \cdot n)$ time in the worst case, where $c$ is the number of birth-death pairs and $n$ is the size of the input complex.
\end{abstract}

\section{Introduction}
\label{sec:introduction}



Simplification of real-valued functions is one of the central topics in Morse theory.
 In the classical (smooth) setting, one of the most notable examples of such simplifications is performed throughout the proof of the h-cobordism theorem~\cite{Milnor:Hcobord, Milnor:morse}.
In the discrete setting, Forman's theory~\cite{Forman1998a, Forman1998b} studies the reversal of a \emph{unique} combinatorial path between two critical points as a~way of reducing the total number of critical points and, therefore, simplifying the Morse complex.

\noindent More recently,
simplification has been studied in the context of persistent homology.
The authors of~\cite{simplification2d} introduced the problem of persistence-sensitive simplification---asking to simplify all pairs with persistence below a given threshold---and gave an algorithm for 2-manifolds. Their solution was later improved to linear time~\cite{Attali2009}. Another approach, based on Forman’s combinatorial vector fields, was presented in~\cite{Bauer:2012aa}. This work drew connection to the cancellation procedure
by observing that whenever the unique path connects two critical points with locally lowest difference in function values,
    their cancellation does not affect the remaining part of the persistence diagram~\cite{Bauer:2012aa}.
    These \emph{apparent pairs} have also been called \emph{close pairs}~\cite{DeRoSh2015}  and \emph{shallow pairs}~\cite{EdelMroz2023}. 

\noindent This observation further helped in optimization of (persistent) homology computation~\cite{Bauer2021} and shape reconstruction~\cite{BauerRoll2024}.
The idea of pruning pairs of critical cells, following Forman's approach, has been extensively studied in data visualization~\cite{ComicDeFloriani2008, CoFlIuMa2016, ImprovedRoad, FeIuFlWe2014, Gunther:2012aa, GyBrHaPa2008, RoWoSh2011}.
Recent works on topological optimization~\cite{GaGaSkGu2020, BigSteps} offer an alternative, albeit less controlled approach to simplification.

\noindent But there remains a critical gap. The works that are able to rigorously control the changes in persistent homology~\cite{simplification2d,Attali2009,Bauer:2012aa} are only able to simplify 0-dimensional persistent homology (as well as codimension-1, by duality). Meanwhile, the middle dimensions---e.g., 1-dimensional homology on 3-manifolds---are important in practice.

\noindent In this work, 
    we study how relations between persistence pairs calculated by the standard lazy reduction algorithm~\cite{ELZ02,Vineyards}
    can guide us in simplifying a discrete Morse function $\mf$, while controlling the changes in its persistence diagram \emph{in any dimension} and the overall gradient structure.
These relations organize persistence pairs in a hierarchical structure called a \emph{depth poset}~\cite{EdLiMrSo2024}.  As observed in~\cite{TransDepth,BigSteps}, these relations describe the obstacles to modifying a function without changing its persistence diagram.

\noindent Concretely, for a given persistence pair $\alpha$, 
    we define \emph{forbidden regions} for its death and birth cells, 
    which describe the parts of the persistence diagram that $\alpha$ cannot move to without changing the persistence pairing.
Conversely, when the forbidden regions leave a gap---a path from $\alpha$ to the diagonal---we can construct a~homotopy
    that brings $\alpha$ to the diagonal without changing the rest of the persistence pairs and the gradient structure.
This allows us to identify a broader family of persistence pairs, possibly with high persistence, that can be safely and selectively removed. 
We summarize our main contribution in the following theorem, 
    where $\bd(\mf)$ denotes the set of birth-death pairs induced by a discrete Morse function $\mf$, and
    $\crit(\cV_{\mf})$, the set of critical cells for $\mf$.

\begin{theorem}\label{thm:main-theorem}
    Let $h$ be a~discrete Morse function on a~Lefschetz complex $X$. 
    If $\bds\in\bd(h)$ is a~persistence pair such that  forbidden regions of its death and birth cells do not intersect,
    and there exists exactly one gradient path between the paired critical cells,
    then there exists a~discrete Morse function $h'$ on $X$ such that 
        $\bd(h')=\bd(h)\setminus\{\bds\}$
        and $\mf(x)=\omf(x)$ for all $x\in\crit(\cV_{\omf})$. 
\end{theorem}


\noindent We present a constructive proof to this theorem,
    which provides an algorithm explicitly tracking all changes in relations throughout the homotopy and the final path reversal. 
As a result, we obtain already computed relations between pairs in $\bd(h')$, which enables iterative simplification.


\section{Preliminaries}
\label{sec:preliminaries}

\begin{definition}\label{def:lefschetz-complex}
    A \emph{Lefschetz complex} is a~triplet $(X,\dim,\mB)$, where $X$ is a~finite set of elements called cells, 
        $\dim: X\rightarrow\mathbb{N}$ 
        is a~map assigning a~dimension to each cell, 
    and $\mB:X\times X\rightarrow \mathbb{Z}_2$ is the boundary coefficient 
    such that
    $\mB(\cx,\cy)\neq 0$ implies $\dim \cx+1=\dim \cy$, 
        in which case we say $\cx$ is a~\emph{facet} of $\cy$.
    Additionally, we require that for any $\cx,\cy\in X$ we have
    $\sum_{z\in X} \mB(\cx,\cz)\cdot\mB(\cz,\cy) = 0$.
    We also define the coboundary coefficient as $\cmB(y,x)\coloneqq\hmB(x,y)$.

\end{definition}

 \noindent Lefschetz complexes generalize simplicial, cubical, and cellular complexes while remaining concrete enough to define persistent homology.
When it does not lead to confusion, we shorten the notation and refer to the set of cells, $X$, as the Lefschetz complex.
We often interpret $\mB$ and $\cmB$ as a~matrix, in which case we put the arguments in the square brackets for emphasis, e.g., $\mB[x, y]$.
We write $X_{n}$ for the set of $n$-dimensional cells of $X$, and $\hmB_n$, $\cmB_n$ for the $n$-th boundary and coboundary matrix, respectively.

\begin{definition}\label{def:dmf}(Discrete Morse function)
    Let $X$ be a~Lefschetz complex.
    A~map $\mf:X\rightarrow\RR$ is called 
    a~\emph{discrete Morse function} ($\dmf$, for short) if the following conditions are satisfied for all $\cx,\cy\in X$.
    \begin{romanenumerate}
        \item if $\hmB(\cx,\cy)=1$ then $\mf(\cx)\leq \mf(\cy)$ (weak monotonicity),
        \item if $\mf(\cx)=\mf(\cy)$ then either $\hmB(\cx,\cy)=1$ or $\hmB(\cy,\cx)=1$ (pairing),
        \item\label{it:dmf_injection} for every $\cy\in\RR$, we have $\card{\mf^{-1}(\cy)}\leq 2$ (almost injective).  
    \end{romanenumerate}
\end{definition}

\noindent In particular, we say that $X$ is \emph{filtered} by $\mf$. 
It also induces an $\mf$-order on $X$:
\[
\cx <_{\mf} \cy \iff \big(\mf(\cx)<\mf(\cy)\big)\ \text{or}\ \big(\mf(\cx)=\mf(\cy)\ \text{and}\ \dim\cx<\dim\cy\big).
\]



\noindent If $X$ is filtered by a $\dmf$ $\mf$, 
    then we always assume that rows and columns of $\hmB_{n}$ are ordered by the $\mf$-order, 
    and those of $\cmB_{n}$ by the reversed $\mf$-order.
    
\noindent To calculate persistent homology, we use the original version of the persistence algorithm~\cite{ELZ02}, called \emph{lazy reduction algorithm}, described in the form we need in \cite{BigSteps}. The algorithm relies on an auxiliary function 
$\low$, which, for a~given column, returns the index of the row containing the lowest non-zero entry in that column. For a~given (co)boundary matrix $\mB_{n}$, Algorithm~\ref{alg:lazy-reduction} performs successive column additions, which results in a~decomposition $\mB_{n}=\mR_{n}\mU_{n}$ with $\mU_{n}$ invertible and upper triangular.  Moreover, if $\cx \neq \cy$ and $\mU_{n}[\cx, \cy] \neq 0$, then column $\mR_{n}[:, \cx]$ was added to $\mR_{n}[:, \cy]$ by the algorithm. Observe that in \(\mD_n\) the rows are indexed by \((n-1)\)-dimensional cells and the columns by \(n\)-dimensional cells.
The same holds for \(\mR_n\); however, both the rows and columns of \(\mU_n\) are indexed by \(n\)-dimensional cells.
Similarly, we obtain $\cmD=\cmR\cmU$ decomposition by applying $\cmD$ to the algorithm.

\begin{algorithm}[ht]
\caption{Lazy reduction of the matrix over $\mathbb{Z}_2$.}\label{alg:lazy-reduction}

  $\mR_n \KwEq \mB_n,\quad \mU_n \KwEq I$\;

\For{$y$ over the columns of $\mR_n$ (left to right)}{%
  \While{$\mR_n[:,\cy]\neq 0$ \KwAnd
         \textup{there exists a~preceding column $x$ with }
         $\low (\mR_n[:,\cx])=\low (\mR_n[:,\cy])$}{
    
    $\mR_n[:,\cy] \gets R_n[:,\cy] + \mR_n[:,\cx]$\;
    
    $\mU_n[\cx,:] \gets \mU_n[\cx,:] +  \mU_n[\cy,:]$\;
  }
}
\end{algorithm}

\noindent We say that $\bds = (\bth\bds, \dth\bds)$ is an ($n$-dimensional) \emph{birth-death} pair if $\bth\bds$ is a~low of $\hmR_{n+1}[:, \dth\bds]$. $\bds$ is an ($n$-dimensional) birth-death pair if and only if $\dth\bds$ is a~low of $\cmR_{n}[:, \bth\bds]$~\cite{dualities}. We refer to $\bth\bds$ and $\dth\bds$ as \emph{birth} and \emph{death} cells, respectively. The \emph{dimension} of a~birth-death pair is the dimension of its birth cell. 
We denote the set of all birth-death pairs by $\bd(\mf)$ and the set of all $n$-dimensional birth-death pairs by $\bd_{n}(\mf)$. The cells in dimension $n$ that are not paired at all---their columns in $\hmR_n$ are zero, and there are no columns in $\hmR_{n+1}$ that have them as the lowest non-zero entry---are \emph{$n$-dimensional homology generators}.
It is convenient to assume that these generators also belong to some birth-death pair, even if its second component is undefined. 

\noindent Let $\cx, \cy \in X_n$. If $\hmU_n[\cx, \cy] = 1$, we say that there is a~\emph{homological relation} between $\cx$ and $\cy$ and denote this fact by $\cx \homo \cy$. Dually, if $\cmU_n[\cx, \cy] = 1$, we write $\cx \cohomo \cy$ to indicate a~\emph{cohomological relation}. If the relation type is not important, we simply write $\cx \rel \cy$. If $\cx$ and $\cy$ are unrelated, we write $\cx \norel \cy$, adding a~superscript to specify the missing relation type, e.g., $\cx \nohomo \cy$ if $U[x,y] = 0$. Observe that if $\cx \homo \cy$, then $\cx$ must be a death cell, whereas $\cy$ may be either a death cell or birth cell. Similarly if $\cx \cohomo \cy$ then $\cx$ has to be a birth cell, while the type of $\cy$ remains unspecified.

\noindent We extend these notions to birth-death pairs: $\bdt \homo \bds$ whenever $\dth\bdt$ is homologically related to any component of $\bds$; similarly for other kinds of arrows. 
As rows and columns of $\hmB_{n}$ and $\cmB_{n}$ are ordered with respect to the $\mf$-order and the reversed $\mf$-order, so are $\hmR_{n}$, $\cmR_{n}$, $\hmU_{n}$ and $\cmU_{n}$. We emphasize that $\cmR_{n}$ and $\cmU_{n}$ are not transposed matrices $\hmR_{n}$ and $\hmU_{n}$, but components of lazy decomposition $\cmB_n = \cmR_{n} \cmU_n$.

\noindent \emph{The persistence diagram} is a~set of two dimensional points $(\mf(\bth\bds),\mf(\dth\bds))$ for $\bds \in \bd(\mf)$. When we visualize a~persistence diagram (see Figure \ref{fig:transposition}), it is convenient to add the \emph{diagonal}, i.e., all points $(x,x)$ for $x \in \RR$, and to annotate the arrows with the type of the relation. Since a~$\dmf$ is not injective in general, it can generate birth-death pairs on the diagonal of the persistence diagram.  We denote the set of such diagonal pairs by $\diagbd(\mf)$ and use notation $\upbd(\mf)$ for the pairs above the diagonal.

\begin{observation}\label{obs:left_up_corner}
    If pairs $\bds,\bdt \in \bd_{n}(\mf)$ and $\bdt \rel \bds$,  then $\mf(\bth\bdt) > \mf(\bth\bds)$ and $\mf(\dth\bdt) < \mf(\dth\bds)$. (see \cite[Lemma 2.1]{BigSteps})
\end{observation}


\noindent The preceding observation means that if we depict every relation between the birth-death pairs of the same dimension as arrows in the persistence diagram, then every such arrow points up and to the left.

\begin{definition}\label{def:topological_simplification}
Let $X$ be a~Lefschetz complex and let $\mf \colon X \to \RR$ be a~dMf.
A \emph{topological simplification of $\mf$} is a~discrete Morse function $\omf$ such that $\upbd(\omf) \subset \upbd(\mf)$ and $\omf(\bth\bds) = \mf(\bth\bds)$ and $\omf(\dth\bds) = \mf(\dth\bds)$, whenever $\bds \in \upbd(\omf)$.
\end{definition}

\noindent In words, a~topological simplification removes some off-diagonal persistence pairs and preserves the rest.

\begin{definition}
     A \emph{combinatorial vector field} (or a~vector field, for short) on a~Lefschetz complex $X$ is a~partition $\cV$ of $X$ into singletons, called \emph{critical cells}, and facet--cofacet pairs, called \emph{vectors}.
    $\crit(\cV)$ denotes the family of all critical cells of $\cV$; $\vect(\cV)$, the family of all vectors. We use the convention that the dimension of a vector is the smaller dimension of its two components.

\end{definition}

    \noindent A combinatorial vector field $\cV$ induces a~digraph $G_\cV=(X,E)$.
    Every edge $(\cx,\cy)\in E$ is either 
        an \emph{explicit arc} when $(\cx,\cy)\in\cV$ or
        an \emph{implicit arc} when $\hmB(\cy,\cx)=1$ and $(\cx,\cy)\not\in\cV$.
    A~path $\rho$ has \emph{dimension} $k$ if it consists only of cells of dimension $k$ and $k+1$.
    In particular, any path from $\cy$ to $\cx$, where $\cx,\cy\in\crit(\cV)$ and $k=\dim(\cx)=\dim(\cy)-1$ is of dimension $k$ and alternates between $k$ and $k+1$ dimensional cells.

    \noindent A combinatorial vector field $\cV$ is called \emph{gradient} if $G_\cV$ is acyclic. If there is a~path between vertices $\cy$ and $\cx$, then we  write $\cy \pathh{\cV} \cx$, omitting the superscript when the vector field is clear from the context. $\cV^{k}$ denotes the union of all $k$-dimensional vectors and critical cells of dimension $k$ and $k+1$. Finally, note that for a given discrete Morse function $\mf$, the non-empty preimages
$
\mathcal V_\mf \coloneqq \{\mf^{-1}(a) \mid a \in \mathbb{R},\ \mf^{-1}(a) \neq \emptyset\}
$
form a combinatorial gradient vector field.


\noindent The \emph{Morse complex} connects homological and dynamical perspectives on scalar functions. It is not required to carry out the reasoning we need, but it will simplify it considerably.  

\begin{definition}[Morse complex]
    \label{def:morse_complex}
    Let $\cV$ be a~combinatorial gradient vector field on $X$. The \emph{Morse complex of $\cV$} is a~Lefschetz complex, denoted by $\cM(\cV)$, consisting of the set of critical cells of $\cV$ along with the restriction of $\dim$. The boundary coefficient $D_{\cM}(\cx, \cy)$ is given by the number of paths in $G_\cV$ from $\cy$ to $\cx$ $\pmod{2}$, provided $\dim \cy = \dim \cx + 1$ and $0$ otherwise.
    
\end{definition}

\noindent The most useful properties of Morse complexes for our work is that they describe the off-diagonal birth-death pairs. Indeed, if $\cV_{\mf}$ is a~gradient vector field of some $\dmf$ $\mf$, its restriction to $\cM(\cV_{\mf})$, denoted by $\mf_{\cM}$, is an injective $\dmf$. 
The next observation follows from \cite[Theorem~4.3]{EdelMroz2023}. 

\begin{corollary}
\label{thm:morse_complex}
    Let $X$ be filtered by $\dmf \; \mf$. Then $\bd(\mf_{\cM}) = \upbd(\mf)$ and $\morsehmU$, $\morsecmU$ are restrictions of $\hmU$ and $\cmU$ to the critical cells.
\end{corollary}

\noindent It follows that we can identify the components of the pairs in $\upbd(h)$ with elements of $\crit(\cV_{\mf})$, while the vectors are the diagonal pairs, $\vect(\cV_{\mf}) = \diagbd(\mf)$. This perspective enables us to apply the following classical theorem.

\begin{theorem}[{\cite[Theorem~9.1]{Forman2002}} ]\label{thm:reversing}
    Let $\cx$ be a~$k$-dimensional critical cell and $\cy$ be a~$k+1$ dimensional critical cell of a~gradient vector field $\cV$. If there exists a~unique path $\rho$ from $\cy$ to $\cx$, then reversing it in $\cV$ produces another gradient vector field, which we denote $\cV^{-\rho}$. The critical cells of $\cV^{-\rho}$ are exactly the critical cells of $\cV$ apart from $\cx$ and $\cy$.
\end{theorem}

\noindent We say that $\bds \in \upbd(\mf)$ is \emph{reversible} if there exists exactly one path between $\dth\bds$ and $\bth\bds$ in $\cV_{\mf}$. However, the theorem alone gives no guarantee that elimination of a~pair of critical cells will not affect the remaining pairs in $\upbd(h)$. Identifying those pairs that can be safely removed is therefore a~key challenge. 

\begin{definition}
Let $X$ be a~Lefschetz complex filtered by $\dmf \; \mf$. 
A pair $(\cx,\cy) \in X \times X$ such that $\hmB(\cx,\cy) = 1$ is a \emph{shallow pair} if $\mf(\cx)$ is the maximum among facets of $y$ and $\mf(\cy)$ is the minimum among cofacets of $\cx$.
\end{definition}

\noindent Observe that every shallow pair is a~birth-death pair. Shallow pairs were introduced as apparent pairs in \cite{Bauer2021} and as close pairs in~\cite{DeRoSh2015}. Since the theory behind them was later developed in the framework of the \emph{depth posets} \cite{TransDepth}, we adopt the name from that setting. Shallow pairs are closely related to an algebraic operation called \emph{Lefschetz cancelation}.

\begin{definition}\label{def:lefschetz-cancellation}
    Let $(s,t)\in X\times X$ be a~pair in a~Lefschetz complex such that $s$ is a~facet of $t$.
    A \emph{cancellation} of $(s,t)$ produces a~\emph{quotient}, another Lefschetz complex $(\hat{X},\hat{\dim},\hat{\mB})$ such that $\hat{X}=X\setminus \{s,t\}$, $\hat{\dim}$ is a~restriction of $\dim$ to $\hat{X}$ and $\hat{\mB}(x,y) = \hmB(x,y) + \hmB(s,y)\cdot\hmB(x,t)$.
\end{definition}

\noindent The boundary map in the quotient can be written in matrix form: if $\dim t = n$ and $\hmBaft_{n}$ is the $n$-th boundary matrix of $\hat{X}$,  then $\hmBaft_{n}[:,\cy] = \hmB_{n}[:,\cy] + \hmB_{n}[s,\cy] \cdot \hmB_{n}[:,t]$, after erasing row $s$ and column $t$. Throughout this paper, we often refer to small modifications of matrices based on their previous state. In such cases, any matrix $M$ after modification is denoted $\hat{M}$.

\begin{theorem}[{\cite[Theorem~3.2]{EdLiMrSo2024}}]\label{thm:cancelation}
   Let $X$ be filtered by a~$\dmf \; \mf$. Fix a~shallow pair $\bds$. Then birth-death pairs of quotient of $X$ after Lefschetz cancelation of $\bds$ are exactly $\bd(\mf)\setminus \{\bds\}$, and every shallow pair of $\mf$ distinct from $\bds$
    remains shallow in the quotient, which may in addition contain new shallow pairs not present in $X$.

\end{theorem}
\noindent In other words, performing a~Lefschetz cancellation on a~shallow pair does not change the pairing between the rest of the cells. It is convenient to characterize shallow pairs in terms of the relations between cells.

\begin{observation}\label{obs:shallow_means_relationless}
    An~$n$-dimensional birth-death pair $\bds$ is shallow if and only if $\hmU_{n+1}[:,\dth\bds]$ and $\cmU_{n}[:,\bth\bds]$ are zero except $\hmU_{n+1}[\dth\bds,\dth\bds]$ and $\cmU_{n}[\bth\bds,\bth\bds]$. Equivalently, $\bds$ is shallow iff $\bdt \norel \bds$ for any birth-death pair $\bdt$.
 
\end{observation}

\noindent It is important to note that a~Lefschetz cancellation leaves intact not only the pairing, but also the relations between cells.

\begin{restatable}{theorem}{LefschetzIsSafe}\label{thm:lefschetz_is_safe}
    Let $\hmB_{n}$ be a~boundary matrix,  
        and $\hmBaft_{n}$, a~boundary matrix of the quotient after cancellation of the~$(n-1)$-th dimensional shallow pair $\bds$.  Let $R_{n}\hmU_{n}$ and $\hat{R}_{n}\hmUaft_{n}$ be their respective decompositions obtained via the lazy reduction. Then, $\hmU_{n}[x,y] = \hmUaft_{n}[x,y]$ for all $x,y$ different than $\dth\bds$. Moreover, symmetrically $\cmU_{n-1}[x,y] = \cmUaft_{n-1}[x,y]$ for all $x,y$ different than $\bth\bds$.
\end{restatable}

\proofappendix

\noindent The above theorem can be rephrased as follows.

\begin{observation}\label{obs:quotient_relation}
    Let $X$ be filtered by a~$\dmf$ and $\bds$ be a~shallow birth-death pair. If $\bdt \rel \bdb$ in $X$, then the same relation holds in the quotient $\hat{X}$ obtained after the cancellation of $\bds$, for all $\bdt \neq \bds$. 
\end{observation}

\noindent Due to Observation~\ref{obs:shallow_means_relationless} above, we introduce \emph{critical shallow pairs}. An off-diagonal birth-death pair $\bds$ is a critical shallow pair if there does not exist an off-diagonal birth-death pair $\bdt$ such that $\bdt \rel \bds$. It is easy to see that critical shallow pairs are exactly the shallow pairs of the Morse complex, although they need not be shallow pairs in the original complex. Equivalently, a pair $\bds$ is critically shallow if for every $\bdt \rel \bds$, the pair $\bdt$ is a vector in $\cV_{\mf}$.
\begin{figure}[ht]
    \centering
    \includegraphics[width=0.7\linewidth]{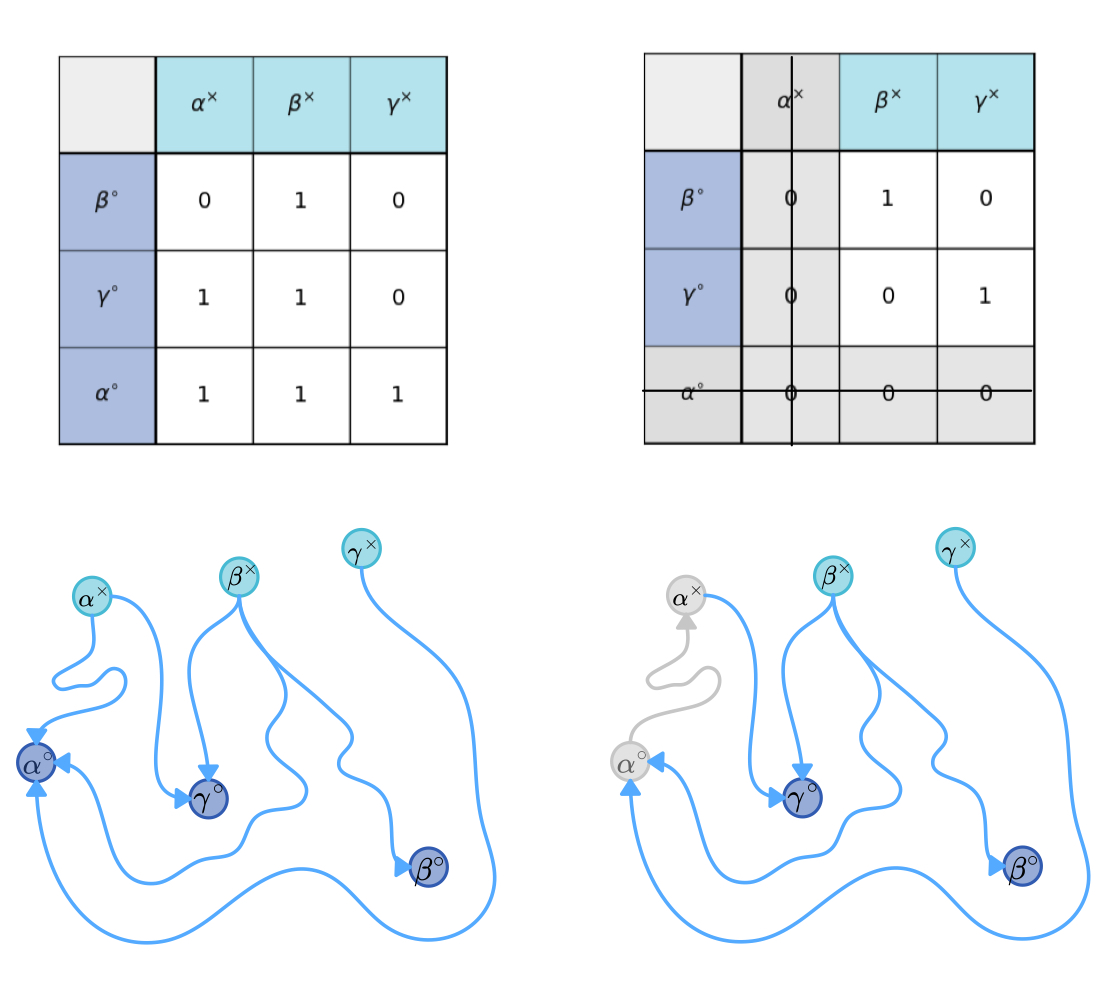}
    \caption{Two vector fields differing by a reversal of the path between components of a birth-death pair $\bds$.
    Critical cells are shown with colored nodes, and arrows between them symbolize paths created by vectors. Above each vector field is the boundary matrix of the corresponding Morse complex.
    Reversing the path between components of $\bds$ gives the same boundary matrix as performing the Lefschetz cancelation.}
    \label{fig:reversing_path}
\end{figure}

\begin{restatable}{theorem}{InversingPathMorseComplex}\label{thm:cancelation_is_path_reversing}
    Let $\cV$ be a~combinatorial vector field on $X$ and let $s, t \in \crit(\cV)$
    be such that $\dim s + 1 = \dim t$.
    Assume that there exists a~unique path $\rho$ from $t$ to $s$. Then $\cM(\cV^{-\rho})$ is isomorphic to the quotient of $\cM(\cV)$ after cancelling the pair $(s, t)$.
    (See example in Figure~\ref{fig:reversing_path}.) 
\end{restatable}

\proofappendix

\noindent So to find a~topological simplification of $\mf$, one can find a~critical shallow pair $\bds$ that is reversible. Then, one has to invert the unique path $\rho$ between $\dth\bds$ and $\bth\bds$, and find $\omf$ with the property that $\cV^{-\rho} = \cV_{\omf}$ and $
\mf_{|\crit(\cV^{-\rho})} = \omf_{|\crit(\cV^{-\rho})}$. Unfortunately, it may happen that there is no pair that is both shallow and reversible. One of the goals of this paper is to remedy this problem.

\section{Homology and cohomology relations in the filter}\label{sec:homo_and_cohomo}
    To understand how birth-death pairs and the relationships between them change during changes of the $\dmf$, one must study how they change upon transposition of two adjacent cells in the boundary matrix. This problem is well-studied; see \cite{Vineyards} and \cite{TransDepth}. Observing that the depth poset can be constructed from the union of homological and cohomological relations between birth-death pairs (see Theorem 4.8 in \cite{EdLiMrSo2024}), we reformulate the results from \cite{TransDepth} in the language of this paper. First, we introduce two additional objects.




\begin{lemma}\cite[Lemma 3.2]{TransDepth}\label{lem:order_not_important}
   Fix a birth-death pair $\bds \in \bd(\mf)$.
    If we remove all birth-death pairs below and to the right of $\bds$---in the region $(\mf(\bth\bds),+\infty]\times[-\infty,\mf(\dth\bds)]$---by iteratively canceling shallow pairs, we get the same boundary matrix regardless of the order of cancellations.
\end{lemma}

\noindent Lemma \ref{lem:order_not_important} proves that the following definition is unambiguous.

\begin{definition}

Fix $\bds,\bdt \in \bd_{n}(\mf)$ such that the components of these pairs are consecutive columns in $\hmB_{n+1}$ or in $\cmB_{n}$.  Define $\mD_{n+1}^{\bds,\bdt}$ to be the matrix obtained by performing Lefschetz cancellations, always canceling shallow pairs, for all pairs lying in the bottom-right quadrant of $\bds$ (excluding $\bdt$ if it eventually lies in this region) and for all pairs lying in the bottom-right quadrant of $\bdt$ (excluding $\bds$ if it eventually lies in this region).

\end{definition}

\noindent Note that in the above definition, we can cancel all pairs in the bottom-right quadrants because, from Observation~\ref{obs:left_up_corner}, there is no $\bdb \in \bd_{n}(\mf)$ such that $\bdt \rel \bdb \rel \bds$.
After the cancellations, we have either (i) both $\bds$ and $\bdt$ are shallow, or (ii) $\bdt \rel \bds$ and $\bdt$ is shallow, or (iii) $\bds \rel \bdt$ and $\bds$ is shallow.

\begin{figure}
    \centering
    \includegraphics[width=0.99\linewidth]{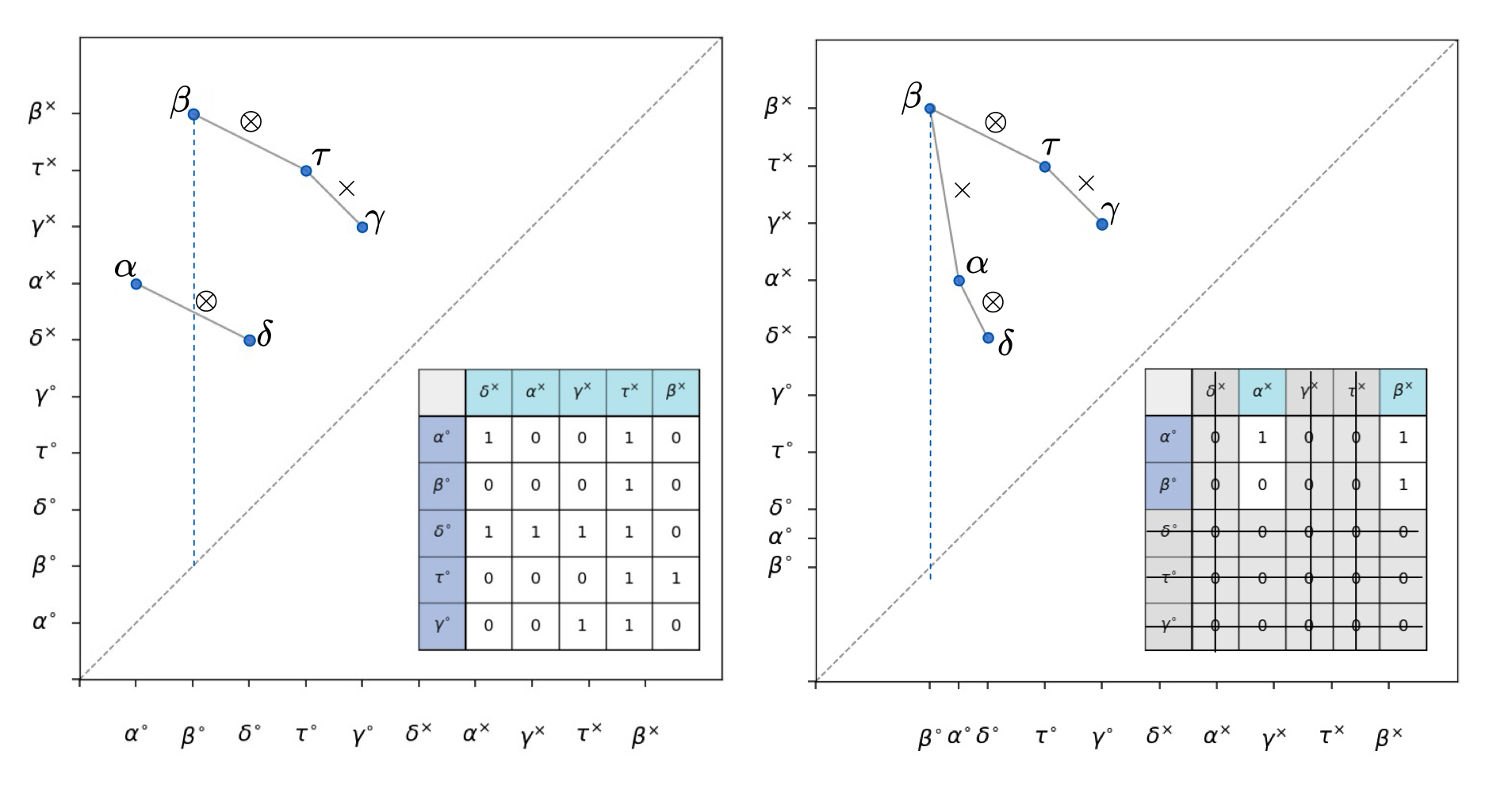}
    \caption{Left: $(n-1)$-st dimensional persistence diagram of some complex $X$, with $\hmB_{n}$ in the bottom-right corner. In the diagram, we denote by $\times$ homological relations between pairs, and by $\otimes$ relations which are homological and cohomological at the same time. To decide if moving $\bth\bdt$ past $\bth\bds$ changes the relations between cells, as determined by \cref{eq:death_death_update} in \cref{thm:birth_transposition}, we need to calculate $\hmB_{n}^{\bds,\bdt}$. Right: The persistence diagram with the updated relation after the transposition of $\bth\bds$ and $\bth\bdt$. In the bottom-right corner, we show $\mD_{n}^{\bds,\bdt}$ before the transposition. The cells deleted by the Lefschetz cancellations are crossed out.}
    \label{fig:transposition}
\end{figure}

\noindent Now we are ready to utilize results from \cite{TransDepth} in a~series of theorems.

\begin{theorem}[Result of death-cells transposition {\cite[Lemma~3.4]{TransDepth}}]\label{thm:death_transposition}
Let $\bds,\bdt$ be $n$-dimensional birth-death pairs such that $\mf(\bth\bds) < \mf(\bth\bdt)$. Then the transposition of $\dth\bds$ and $\dth\bdt$ does not change the values in $\hmU_{n+1}$, while the changes in $\cmU_n$ follow these rules:

\begin{bracketenumerate}

        \item If $\bdt \nohomo \bds$ and $\left(\bdt \cohomo \bds \text{ or } \mD_{n+1}^{\bds,\bdt}[\bth\bds,\dth\bdt]=1\right)$, then the pairing is unaffected and the row $\bth\bdt$ of the matrix $\cmU_{n}$ changes according to the formula:
    \begin{equation}\label{eq:death_death_update}
        \cmUaft_n[\bth\bdt,:] = \cmUbef_{n}[\bth\bdt,:]+\cmUbef_{n}[\bth\bds,:],
    \end{equation}

\item  If $\bdt \homo \bds$, then the pairs $(\bth\bds,\dth\bds), (\bth\bdt,\dth\bdt)$ turn into $(\bth\bds,\dth\bdt)$ and $(\bth\bdt,\dth\bds)$ and $\cmU_{n}$ changes as in \eqref{eq:death_death_update}.

\item  Otherwise, $\cmU_{n}$ and the pairing remain unchanged.

\end{bracketenumerate}
\end{theorem}

\begin{theorem}[Result of birth-cells transposition {\cite[Lemma~3.3]{TransDepth}}]\label{thm:birth_transposition}

Let $\bds$ and $\bdt$ be $n$-dimensional birth-death pairs such that $\mf(\dth\bdt) < \mf(\dth\bds)$. Then the transposition of $\bth\bds$ and $\bth\bdt$ does not change the values in $\cmU_{n}$, while the changes in $\hmU_{n+1}$ follow these rules:

\begin{bracketenumerate}
 
 \item  If $\bdt \nocohomo \bds$ and $\left( \bdt \homo \bds~\textrm{or}~\mD_{n+1}^{\bds,\bdt}[\bth\bdt,\dth\bds] = 1 \right)$, then the pairing is unaffected and the row $\dth\bdt$ of the matrix $\hmU_{n+1}$ changes according to the formula:
    \begin{equation}\label{eq:birth_birth_update}
          \hmUaft_{n+1}[\dth\bdt,:] = \hmUbef_{n+1}[\dth\bdt,:]+\hmUbef_{n+1}[\dth\bds,:]
    \end{equation}

\item  If $\bdt \cohomo \bds$, then $\hmU_{n+1}$ changes as in $\eqref{eq:birth_birth_update}$, and pairs $(\bth\bds,\dth\bds), (\bth\bdt,\dth\bdt)$ turn into $(\bth\bdt,\dth\bds)$ and $(\bth\bds,\dth\bdt)$.

\item Otherwise $\hmU_{n+1}$ and the pairing remain unchanged.

    \end{bracketenumerate}
\end{theorem}

\begin{theorem}[Result of birth-death transposition {\cite[Lemma~3.5]{TransDepth}}]\label{thm:mixed_transposition}
A transposition between a~birth and a~death cell, which is a~result of increasing birth, or decreasing death does not affect pairing or relations between birth-death pairs. 
\end{theorem} 

\noindent Figure \ref{fig:transposition} presents an example of how a transposition affects the relationship between birth-death pairs.
Now we introduce our own propositions, which will be useful later.


\begin{restatable}{proposition}{WeCanOnlyLose}\label{prop:we_can_only_lose}
    Fix a~pair $\bds \in \bd_{(n-1)}(\mf)$. 
    A transposition that increases the value of $\bth\bds$  or decreases the value of $\dth\bds$ and does not cause a~switch cannot create a~relation $\bdt \rel \bds$ for any pair $\bdt$.
\end{restatable}

\noindent Note that a transposition may involve skipping two columns and rows when bypassing a combinatorial vector. The following proposition helps decrease complexity of the final algorithm.

\begin{proposition} \label{prop:vectors_not_create_critical_relations}
    Let $\mf, \omf$ be two $\dmf$s such that $\cV_{\mf} = \cV_{\omf}$ and the difference between $\mf$-order and $\omf$-order is a~transposition between a~critical cell and a~vector. Then $\mf$ and $\omf$ generate the same off-diagonal birth-death pairs and relations between them.
\end{proposition}

\begin{proof}
    As this process does not change the (co)boundary matrix of $\cM(\cV_{\mf})$, it cannot change the pairing or relations between critical cells.
\end{proof}

\noindent Finally, the following two corollaries give us an opportunity to focus only on specific cases during the construction of the homotopy below.

\begin{corollary}\label{cor:bd_pairs}
    Take a~pair $\bds \in \bd_{n}(\mf)$.  If $\cx$ is a~cell such that $\bth\bds <_{\mf} \cx$ and also $\cx \cohomo \bth\bds$, then $\cx$ is an $n$-dimensional birth cell. Analogously, if $\cy <_{\mf} \dth\bds$ and $\cy\homo\dth\bds$, then $\cy$ is an $(n+1)$-dimensional death cell.

\end{corollary}

\begin{proof}
    Because $\cy \homo \dth\bds$, the column indexed by $\cy$ was added to column $\dth\bds$ during the lazy reduction of matrix $\hmB_{n+1}$. Because lazy reduction never adds zero columns, column $\cy$ in $\hmR_{n+1}$ has a~unique low, so it is a~death cell. Analogously, if $\cx \cohomo \bth\bds$, then column $\cx$ was added to column $\bth\bds$ in $\cmB_{n}$, so $\cx$ is a~birth cell.
\end{proof}

\section{Constructing the homotopy}

\subsection{Homotopy}\label{sec:homotopy}
Recall that a~linear homotopy between two maps $f_0$ and $f_1$ is a~family of maps $f_t(x):=H(t,x) = (1-t)f_0(x) + tf_1(x)$ for $t \in [0,1]$.

\noindent We say that $A \subset X$ is \emph{connected} if it Hasse diagram---the graph whose vertices are the cells of $A$ with an edge for every boundary relation---is connected.
An $f$-\emph{induced partition} is a~partition $\cA$ of $X$ into maximal, with respect to inclusion, sets $A$, such that $f$ is constant on $A$, and every $A$ is connected.

\begin{restatable}{theorem}{DmfHomotopy}\label{thm:dmf-homotopy}

    Let $f_0$ and $f_1$ be two $\dmf$s defined on $X$ such that $\cV_{f_0} = \cV_{f_1}$.
    Let $f_t(\cx) := H(t,\cx)$ be the linear homotopy between $f_0$ and $f_1$.
    Let $\cV_{f_t}$ denote the $f_t$-induced partition of~$X$.
    Then, $\cV_{f_t} = \cV_{f_0}$  
       for every $t\in[0,1]$. 

\end{restatable}

\proofappendix

\noindent Using this theorem, we can represent our homotopy as a~finite series of transpositions, allowing us to analyze only a~finite number of time steps.
Indeed, along a homotopy $(f_t)_{t\in[0,1]}$ there are only finitely many parameters $t$ at which $f_t$ fails to be a $\dmf$.
On each open interval between two such parameters, the induced $f_t$-order is well-defined and remains constant (in particular, it does not depend on $t$).
Consequently, for a sufficiently fine discretization of $[0,1]$, consecutive $f_t$-orders differ by exactly one transposition.

\subsection{Journey to the diagonal}\label{sec:evade}

Consider an example in Figure~\ref{fig:vk_and_PD} and assume that our goal is to reduce the lifetime of the pair $\bds$ to be arbitrarily small, without changing the pairing or the vector field. 
To reduce the lifetime, we may increase the value of $\bth\bds$ and decrease the value of $\dth\bds$, along with a~set of vectors. We        may implement this as a~series of ``moves'' of the birth-death pair to the right and down in the persistence diagram.

\noindent Unfortunately, our moves are constrained:
if we want to preserve the original vector field, then we cannot decrease $\dth\bds$ below $\bth\bdt$ as $\dth\bds \pathh{} \bth\bdt$, and similarly, $\bth\bds$ cannot increase above~$\dth\bdb$. Moreover, as $\dth\bdt \homo \dth\bds$ and $\dth\bdb \homo \dth\bds$, we also cannot decrease $\dth\bds$  below these levels, without switches in pairing. Even worse, because $\bth\xi \cohomo \bth\bds$, $\bth\bds$ cannot increase above $\bth\xi$ without another switch.

\noindent This appears to be a~serious obstacle. However, when we examine the persistence diagram (see bottom part of the  Figure~\ref{fig:vk_and_PD}), we notice, following Observation~\ref{obs:left_up_corner}, that increasing $\bth\bds$ above $\bth\bdt$ breaks both homological relations of $\dth\bds$ without changing the pairing. 
Afterwards, we are able to decrease $\dth\bds$ as close to $\bth\bds$ as we want. This motivates our central notion of forbidden regions, which describe the allowed ``moves'' in the persistence diagram.

\begin{definition}[Forbidden regions]
    For an off-diagonal pair $\bds \in \upbd(\mf)$, we say that:

    \begin{bracketenumerate}

        \item  \emph{Forbidden region for} $\dth\bds$ is defined as
        \[\fhR_{\mf}(\bds):= \bigcup\limits_{\substack{\bdt \homo \bds \\ \bdt \in \upbd_{n}(\mf)}}[-\infty, \mf(\bth\bdt)] \times [-\infty, \mf(\dth\bdt)] \;\;\cup \bigcup\limits_{\substack{\dth\bds \pathh{} \cx \\ \cx \in \crit(\cV_{\mf}) }}[-\infty,\mf(x)] \times[-\infty,\mf(x)].\]

        \item \emph{Forbidden region for} $\bth\bds$ is defined as
        \[\fcR_{\mf}(\bds) := \bigcup\limits_{\substack{\bdt \cohomo \bds \\ \bdt \in \upbd_{n}(\mf)}}[\mf(\bth\bdt), +\infty] \times [\mf(\dth\bdt), +\infty] \;\; \cup \bigcup\limits_{ \substack{\cy \pathh{} \bth\bds \\ \cy \in \crit(\cV_{\mf})}}[\mf(y),+\infty] \times[\mf(y),+\infty].\] 

    \end{bracketenumerate}
\end{definition}

\noindent Once we have the notion of forbidden regions, we can define a~set of safe transformations, which we call \emph{allowed moves}.

\begin{definition}[Allowed moves]
    Let $\mf$ be a $\dmf$, $\bds\in\upbd(\mf)$ and $\cc\in\bds$.
    A~\emph{pre-allowed move} of $\bds$ is a~new $\dmf$ $\omf$ such that:
    \begin{enumerate}[label=(\arabic*)]
        \item $\cV_{\mf} = \cV_{\omf}$ and for all $x\in\crit(\cV_\mf)\setminus\{c\}$ we have $\mf(x)=\omf(x)$,
        \item If $\cc$ is a~birth cell, then $\omf(\cc) > \mf(\cc) $; if $\cc$ is a~death cell, then $\omf(\cc) < \mf(\cc) $,
        \item\label{it:single-transposition} $\mf$-order and $\omf$-order restricted to $\crit(\cV_{\mf})$ differ by a single transposition at most.
    \end{enumerate}
   If a~pre-allowed move $\omf$ is such that $\bd(\mf) = \bd(\omf)$, then we say that $\omf$ is an \emph{allowed move}.
\end{definition}

\noindent A pre-allowed move pushes the pair $\alpha$ containing $\cc$ toward the diagonal either by increasing birth or decreasing death without affecting the vector field.
A single pre-allowed move bypasses at most one other critical cell. 
We note that multiple vectors can change their value and position in the $\omf$-order---as long as the gradient structure is preserved.
We will use the allowed moves to construct the homotopy bringing a persistent pair to the diagonal.

\begin{corollary}\label{cor:vectors_are_unimportant}
    If $\omf$ is a~pre-allowed move for $\mf$, then
    the change in persistence pairing can only result from transpositions of critical cells.
\end{corollary}
\begin{proof}
     The linear homotopy from $\mf$ to $\omf$ may be expressed as a~series of transpositions in $\mf$-order, given by specific times $t \in [0,1]$ and $\mf_{t}$-orders. 
    By Theorem~\ref{thm:dmf-homotopy}, the transpositions do not change the vector field, and thus, the diagonal pairs.
    Therefore, by Proposition~\ref{prop:vectors_not_create_critical_relations}, the change in persistence pairing can only result from transpositions of critical cells.
\end{proof}

 \noindent Observe that for an $X$ filtered by $\dmf \; \mf$, for every cell $x$ and interval $[a,b]$ such that $\mf(x) \in [a,b]$, we can find $t \in [0,1]$ such that $\mf(x) = at + (1-t)b$. We call it the \emph{linear coefficient of} $x$ on $[a,b]$.
We now show that, for a~fixed $\mf$, one can construct a~pre-allowed move that pushes the chosen birth-death pair to the right, and another one that pushes it downward.

\begin{restatable}[Increasing birth -- moving right]{proposition}{MoveRightProp}\label{prop:pre_allowed_move_right}
Let $X$ be filtered by a~$\dmf$ $\mf$. Let $\bds \in \upbd(\mf)_{k}$ be an off-diagonal pair, and $\delta,\xi$ be real values such that $\mf(\bth\bds) < \delta < \xi < \mf(\dth\bds)$, and there is at most one $\ce \in \crit(\cV_{\mf})$ such that $\mf(\ce) \in (\mf(\bth\bds),\delta)$. Additionally, assume that $\ce \notpathh{} \bth\bds$ and $\mf^{-1}([\delta,\xi]) = \emptyset$. Define

\[
\omf(x)=
\left\{
\begin{array}{c@{\quad}l}
t_x\delta+(1-t_x)\xi & \text{when }\mf(x)\in [\mf(\bth\bds),\xi] \text{ and } \cx \pathh{\cV_{\mf}}   \bth\bds \text{ and } x \not \in \crit(\cV_{\mf})\setminus\{\bth\alpha\},\\[6pt]
h(x) & \text{otherwise,}
\end{array}
\right.
\]
where $t_x$ is the linear coefficient of $x$ on the interval $[\mf(\bth\bds),\xi]$. Then $\omf$ is a~pre-allowed move of $\bds$ with respect to $\mf$.
    
\end{restatable}

\proofappendix

\begin{figure}[H]   
  \centering
  \includegraphics[width=0.85\linewidth]{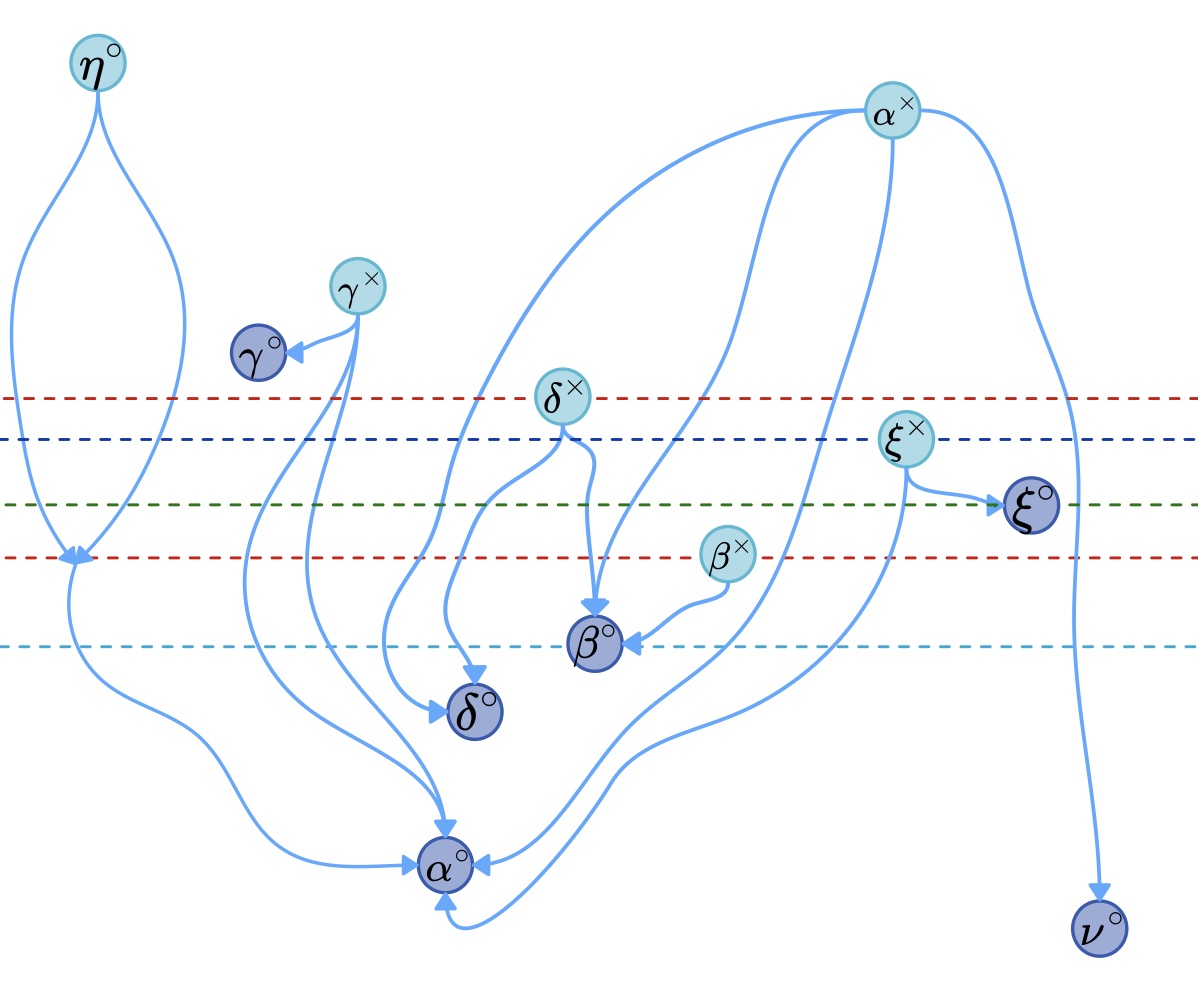}

  \vspace{6pt}

  \includegraphics[width=0.5\linewidth]{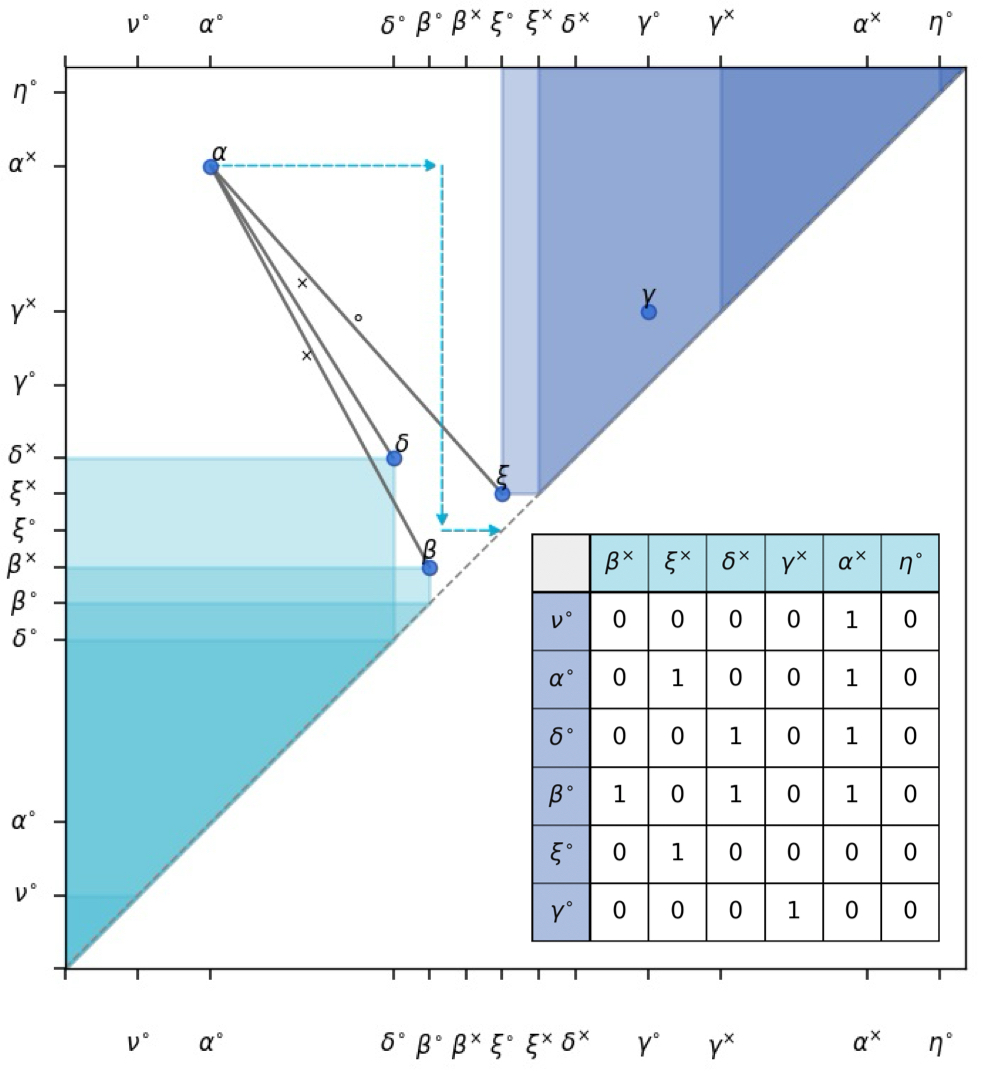}

  \caption{%
Top: Schematic picture of $\cV^{k}_{\mf}$. Node heights encode values of $\dmf$, critical cells are labeled by Greek letters with superscripts. Several important sublevels are highlighted with dashed lines.
Bottom: Boundary matrix and the persistence diagram of the Morse complex induced by $\cV^{k}_{\mf}$ with the two kinds of forbidden regions highlighted, and relations involving the birth-death pair $\bds$ shown as edges. The forbidden regions for $\dth\bds$ are shown in light blue; those for $\bth\bds$, in darker blue. The dashed arrows illustrate a possible homotopy, which moves the point to the diagonal.%
}
  \label{fig:vk_and_PD}
\end{figure}

\begin{restatable}[Decreasing death -- moving down]{proposition}{MoveDownProp}\label{prop:pre_allowed_move_down}

Let $X$ be filtered by $\dmf \; \mf$. Let $\bds \in \upbd(\mf)_{k}$ be an off-diagonal pair, and $\xi,\delta$ be real values such that  $ \mf(\bth\bds) < \xi < \delta < \mf(\dth\bds)$, and there is at most one $\ce \in \crit(\cV_{\mf})$ such that $\mf(\ce) \in (\delta, \mf(\dth\bds))$. Additionally, assume $\dth\bds \notpathh{} \ce$ and at the same time $\mf^{-1}([\xi,\delta]) = \emptyset$. 
Define 

\[
\omf(x)=
\left\{
\begin{array}{c@{\quad}l}
t_x\xi+(1-t_x)\delta & \text{when }\mf(x)\in [\xi,\mf(\dth\bds)] \text{ and } \dth\bds \pathh{\cV_{\mf}}   \cx \text{ and } x \not \in \crit(\cV_{\mf})\setminus\{\dth\alpha\},\\[6pt]
h(x) & \text{otherwise,}
\end{array}
\right.
\]
where $t_x$ is the linear coefficient of $x$ on the interval $[\xi,\mf(\dth\bds)]$. Then $\omf$ is a~pre-allowed move of $\bds$ with respect to $\mf$.
    
\end{restatable}

\proofappendix

\noindent Now observe that an allowed move of $\bds$ does not introduce new forbidden regions.

\begin{lemma}\label{lem:we_can_only_lose_regions}
    Let $\omf$ be an allowed move of $\bds \in \upbd(\mf)$. Then $\fhR_{\omf}(\bds) \subset \fhR_{\mf}(\bds)$ and $\fcR_{\omf}(\bds) \subset \fcR_{\mf}(\bds)$.
\end{lemma}

\begin{proof}
    The statement follows directly from Proposition~\ref{prop:we_can_only_lose} and the fact that $\cV_{\mf} = \cV_{\omf}$, and we are changing the value of only one component of $\bds$.
\end{proof}

\noindent It follows that if we know the initial forbidden regions, we can design a sequence of allowed moves that brings $\bds$ arbitrarily close to the diagonal.

\begin{theorem}\label{thm:when_you_can_move_to_diagonal}
    Let $X$ be filtered by a~$\dmf$ $\mf_{1}$ and $\bds \in \upbd_{k}(\mf)$ be such that $\fcR_{\mf_1}(\bds) \cap \fhR_{\mf_1}(\bds) = \emptyset$, then there exists a~sequence of $\dmf$s $\mf_{1}, \mf_{2}, \ldots, \mf_{n}$ such that $\mf_{i+1}$ is an allowed move of $\bds \in \bd(\mf_{i})$, and $\mf_{n}(\dth\bds) - \mf_{n}(\bth\bds)$ is arbitrarily small. 
\end{theorem}

\begin{proof}

    We begin by showing that we are able to construct from $\mf_{i}$ an allowed move $\mf_{i+1}$ such that $\bth\bds$ and $\dth\bds$ are closer in  $\mf_{i+1}$-order than in $\mf_{i}$-order, where both orders are restricted to the critical cells. Due to Corollary~\ref{cor:vectors_are_unimportant}, to show that $\bd(\mf_{i}) = \bd(\mf_{i+1})$, we can focus only on transpositions between critical cells. Define $\cx$ and $\cy$ as critical cells such that in $\mf_{i}$-order, we have $\bth\bds <_{\mf_{i}} \cx <_{\mf_{i}} ... <_{\mf_{i}} \cy <_{\mf_{i}} \dth\bds $. 

    \noindent If $\cx \notpathh{} \bth\bds$, then if $\cx$ is a~death cell or $\cx \nocohomo \bth\bds$, we use Proposition~\ref{prop:pre_allowed_move_right} to construct $\mf_{i+1}$, which increases the value of $\bth\bds$ with the values of $\delta$ and $\xi$ in the proposition larger than $\mf_{i}(\cx)$. Analogously, if $\dth\bds \notpathh{} \cy $, then if $\cy$ is a~birth cell or $\cy \nohomo \dth\bds$, then use Proposition~\ref{prop:pre_allowed_move_down} to construct $\mf_{i+1}$, which decreases the value of $\dth\bds$ to bypass $\cy$. From Theorems~\ref{thm:death_transposition},
    \ref{thm:birth_transposition} and \ref{thm:mixed_transposition}, we get that these are indeed allowed moves.

    \noindent If $\cx \pathh{} \bth\bds$ and $\dth\bds \pathh{} \cy$, then $\cx$ generates forbidden regions bounded by a~vertical line and $\cy$ generates  forbidden region bounded by a~horizontal line. Because $\cx <_{\mf_{i}} \cy$, they intersect.
    We get the same argument if $\cx \pathh{} \bth\bds$ and $\cy \homo \dth\bds$,  or $\dth\bds \pathh{} y$ and $\cx \cohomo \bth\bds$. If $\cx \cohomo \bth\bds$ and $\cy \homo \dth\bds$, then due to Corollary~\ref{cor:bd_pairs}, there has to exist $(x,\dth\bdt), (\bth\bdb,y) \in \bd_{k}(\mf_i)$. It follows from Observation~\ref{obs:left_up_corner} that they are in the bottom-right quadrant of the pair $\bds$. Therefore, they generate forbidden $\dth\bds$ and $\bth\bds$ regions, which intersect.

    \noindent It follows from Lemma~\ref{lem:we_can_only_lose_regions} that if $\fhR_{\mf_{i}}(\bds)\cap\fcR_{\mf_{i}}(\bds) \neq \emptyset$, then $\fhR_{\mf_{1}}(\bds)\cap\fcR_{\mf_{1}}(\bds) \neq \emptyset$.

    \noindent Accordingly, we can construct a~series of allowed moves, such that $\mf_{j}$ is the last one, and $\bth\bds$ and $\dth\bds$ are consecutive in the $\mf_{j}$-order restricted to the critical cells. Then, by Proposition~\ref{prop:pre_allowed_move_right}, we construct a~final pre-allowed move such that the value gap between the cells of $\bds$ can be made arbitrarily small. Since no critical cell is bypassed  during this deformation, the move is allowed.
\end{proof}

\subsection{Reversing the path}\label{sec:reversing_the_path}

In the previous subsection, we showed that if the forbidden regions of the birth and the death cell of a~(reversible) pair $\bds$ do not intersect, then we can reduce its lifetime arbitrarily close to zero. 
In particular, we can make it a critical shallow pair. It follows from Theorem~\ref{thm:cancelation_is_path_reversing} that we can safely---that is, without introducing changes in the pairing or in the relations---reverse the path between the components of $\bds$. 
The reversal is the final step of the construction, which corresponds to $\alpha$ entering the diagonal.
To make the homotopy fully explicit we construct the final $\dmf$ inducing the vector field with reversed path.

\begin{restatable}{proposition}{RevPathProp}\label{prop:reversing_path}
        Let $\bds \in \bd_{n}(\mf)$ be a~reversible pair such that the unique path between $\dth\bds$ and $\bth\bds$ is $\rho$. If $\mf^{-1}([\mf(\bth\bds),\mf(\dth\bds)]) = \rho = (\bth\bds = x_0, x_1, x_2, \ldots , x_m = \dth\bds)$, then the function $\omf$, defined as
    \[
\omf(x)=
\left\{
\begin{array}{c@{\quad}l}
 \mf(x_{m - 2 \lfloor i/2 \rfloor})& \text{when } x \in \rho \text{ and } x = x_i,\\[6pt]
h(x) & \text{otherwise,}
\end{array}
\right.
\]
is a~$\dmf$ which generates $\cV^{-\rho}$ and does not change the value of the critical cells other than the components of $\bds$.
\end{restatable}

\section{Final algorithm and summary}\label{sec:algorithm}

\subsection{Final algorithm}

We summarize the entire construction in the form of an algorithm that produces a~topological simplification of a given $\dmf$.
\vspace{1ex}

\noindent
\textbf{Input:} A Lefschetz complex $X$ filtered by a~$\dmf$ $\mf$; the combinatorial vector field $\cV_{\mf}$; the set $\bd(\mf)$ of birth–death pairs; all homology/cohomology relations among the off-diagonal pairs;
and a~reversible, $k$-dimensional birth-death pair $\bds$, such that $\fhR_{\mf}(\bds) \cap \fcR_{\mf}(\bds)
  = \emptyset. $


\begin{bracketenumerate}
    

    \item Following the procedure described in Theorem~\ref{thm:when_you_can_move_to_diagonal}, move $\bds$ so close to the diagonal that $\mf^{-1}((\mf(\bth\bds),\mf(\dth\bds)) = \rho$, where $\rho$ is a~unique path between the components of $\bds$. During this process, update the relations between the critical cells of $\cV_{\mf}$ using Theorems \ref{thm:death_transposition}, \ref{thm:birth_transposition}.\label{step:move}
    
    \item Reverse the path $\rho$ between $\dth\bds$ and $\bth\bds$ in the vector field, constructing a~new $\dmf$ as described in Proposition \ref{prop:reversing_path}.\label{step:reverse}

\end{bracketenumerate}

\noindent
\textbf{Output:} A Lefschetz complex $X$ filtered by a~$\dmf$ $\omf$; the combinatorial vector field $\cV_{\omf}$; the set $\bd(\omf)$ of birth-death pairs; all homology/cohomology relations among the off-diagonal pairs.

\noindent Figure~\ref{fig:result} illustrates an example of a topological simplification obtained by this procedure.

\begin{theorem}\label{thm:algorithm_works}
    The algorithm above returns a~topological simplification $\omf$ for a~$\dmf$ $\mf$.
    Moreover, the output of the algorithm contains the updated vector field and homology/cohomology relations for $\omf$.
\end{theorem}

\begin{proof}
    We start by showing that $\upbd(\omf) = \upbd(\mf)\setminus\{\bds\}$. Step~$\ref{step:move}$ does not cause any changes in the pairing by Theorem~\ref{thm:when_you_can_move_to_diagonal}. A $\dmf$ constructed in Step~\ref{step:reverse} has the same off-diagonal pairs as the previous one, except $\bds$, due to Theorem~\ref{thm:cancelation_is_path_reversing} and Theorem~\ref{thm:cancelation}.

\noindent Proposition~\ref{prop:vectors_not_create_critical_relations} and Theorem \ref{thm:cancelation_is_path_reversing} imply that it suffices to apply the update pattern and check its conditions only during transpositions of critical cells in Step~\ref{step:move}. This results in the updated relations among critical cells at the end of the algorithm.

\noindent After the update, the new vector field $\cV_{\omf} = \cV^{-\rho}$, where $\rho$ is a~unique path. We know the birth-death pairs $\bd(\omf)$, as well as all relations between critical cells after the application of the update patterns.
\end{proof}

\noindent The proof of Theorem \ref{thm:algorithm_works} also proves Theorem \ref{thm:main-theorem}. 
For an iterative execution of the algorithm, we can use its output as an input for the next run; one only needs to provide the next eligible pair.
Finally, we consider how much the new constructed $\dmf$ $\omf$ differs from the original one.

\begin{proposition}
    Let $X$ be filtered by a~$\dmf$ $\mf$, and $\omf$ be its topological simplification constructed by our algorithm, which removes pair $\bds$. Then, the difference between $\omf$ and $\mf$ is bounded by the lifetime of $\bds$, that is 
    $\max\limits_{x \in X} |\mf(\cx) - \omf(\cx)| \leq (\mf(\dth\bds) - \mf(\bth\bds))$.
\end{proposition}

\begin{proof}
    It follow directly from the fact that in Propositions \ref{prop:pre_allowed_move_right}, \ref{prop:pre_allowed_move_down} and \ref{prop:reversing_path}, we change only the values of the cells between $\mf(\bth\bds)$ and $\mf(\dth\bds)$, and if the value of the $\dmf$ is changed on $\cx$, then the resulting value also lies between $\mf(\bth\bds)$ and $\mf(\dth\bds)$.
\end{proof}

\subsection{Complexity}
We note that checking if a pair $\bds$ can serve as an input to the algorithm takes $\mathcal{O}(n \log n)$ time, see Appendix~\ref{sec:complexity}; there are at most $c$ such pairs to check.
The complexity of the algorithm is dominated by the cost of checking if moving pair $\bds$ past pair $\bdt$ requires updating relations between birth-death pairs, whenever $\bdt \norel \bds$.
Computing \(\mD_{n+1}^{\bds,\bdt}\) can clearly be done in \(\mathcal{O}(n^2)\) time; however, in the Appendix \ref{sec:complexity} we show that it can be reduced to \(\mathcal{O}(n)\).
 Therefore, the worst case running time is $\mathcal{O}(c \cdot n)$, where $c$ is the number of birth-death pairs, and $n$ is the number of cells in the complex.

 \subsection{Summary}

We presented a~new criterion for removing a~fixed birth-death pair.
We have also shown that for every pair that satisfies this criterion, it is possible to construct a~homotopy, which moves this pair into the diagonal.
The paper opens a~number of questions.

\begin{bracketenumerate}
    \item How does the order of cancellations affect the possibility of canceling the remaining pairs? Is there an optimal order? Can we find a~hierarchy of cancellations using this order?
    \item Is the criterion exhaustive? That is, are there other removable pairs that are not captured by the criterion?
    \item Is it possible to weaken the criterion by proper manipulation of the pairs generating the forbidden regions? For example, if forbidden region of $\bth\bds$ and forbidden region of $\dth\bds$ intersect, is it possible to manipulate other cells to clear a~path to the diagonal for $\bds$, and restore their values after the cancellation?
    \item Is it possible to parallelize the cancellation process? If so, for which pairs?
\end{bracketenumerate}

\begin{figure}[H]
  \centering
  \includegraphics[width=0.83\linewidth]{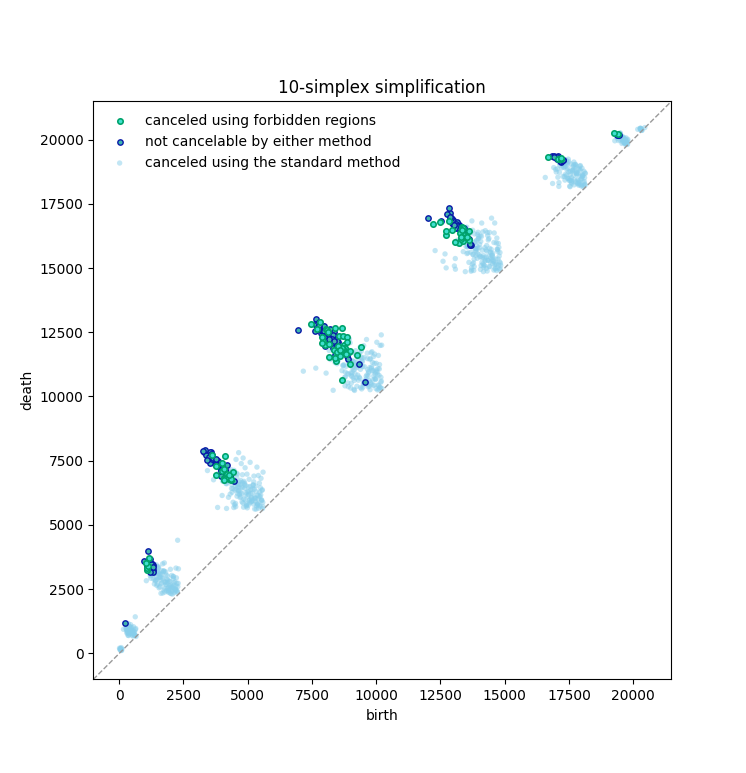}

  \caption{
Persistence diagram of the 10-simplex, filtered by a random injective $\dmf$ such that every birth-death pair of dimension $n$ is separated from pairs of dimensions $n+1$ and $n-1$. We apply a procedure that first simplifies $\dmf$ by the standard method, i.e., path reversing between shallow pairs. When there is no reversible shallow pair left, we continue using the algorithm described in this paper. We made multiple passes canceling any pair that met the algorithm’s assumptions. We stopped when there was no reversible pair with a path between forbidden regions. Pairs of different types (canceled by the standard method, canceled using forbidden regions, not cancelable) are denoted by different colors. Figure \ref{fig:result_zoom} zooms-in on the pairs in dimension 4.}
  \label{fig:result}
\end{figure}

\begin{figure}[ht]
    \centering
    \includegraphics[width=1\linewidth]{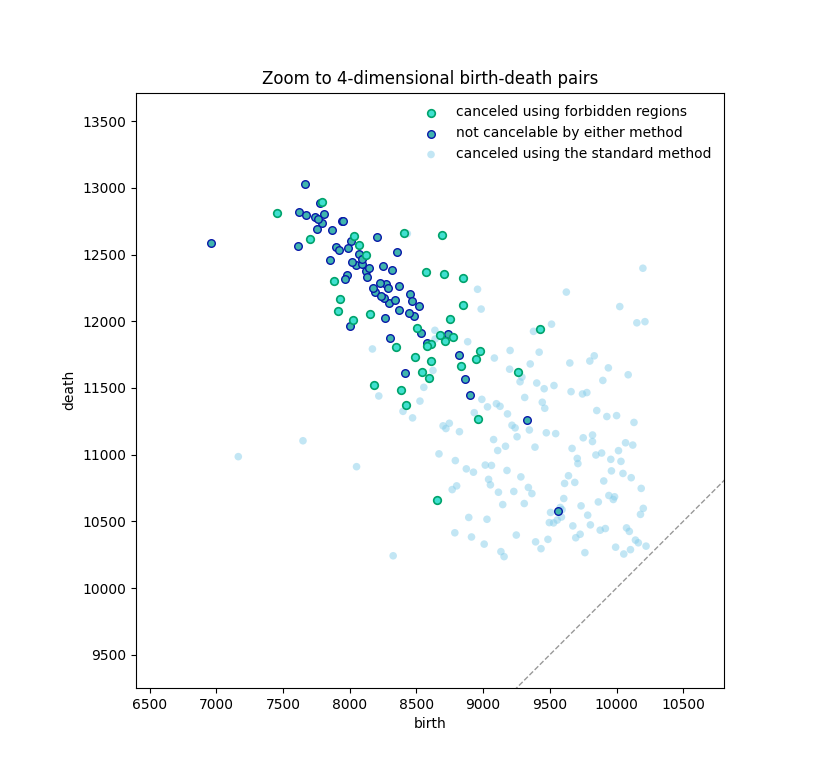}
    \caption{Birth-death pairs in dimension 4 from Figure \ref{fig:result}.}
    \label{fig:result_zoom}
\end{figure}

\subsection{Acknowledgements}
\noindent Jakub Leśkiewicz wants to thank his supervisor, Prof. Marian Mrozek, for
scientific guidance, patience, and opportunity to delay the rest of his duties while writing this work.
The author also extends thanks to his entire family, to Zuzanna Świątek, and to Mikołaj Kardyś,
BEng, MSc, for providing meals during the most intensive periods of work.





\bibliography{bibliography}

\appendix

\begin{appendices}

\section{Technical proofs}\label{sec:technical}

\LefschetzIsSafe*

    \begin{observation}\label{obs:column-in-R-as-a-sum-of-columns-in-D}
    Every column $\hmR_n[:,z]$ (and $\hat{\hmR}_n[:,z]$, respectively) is a sum of column $\hmB_n[:,z]$ and columns to the left of $z$, 
        which are already reduced at the moment of addition. 
    Moreover, there exist columns that do not need any reduction, that is, $\hmB_n[:,x] = \hmR_n[:,x]$. 
    Recursively, every column $\hmR_n[:,z]$ can be expressed as a sum of columns in $\hmB_n$ in an unambiguous way.
    In particular, it is directly given by matrix $V_{n} =  \hmU_n^{-1}$, which can also be computed during the lazy reduction (see \cite{Vineyards,BigSteps}).
    \end{observation}

\begin{proof}[Proof of Theorem \ref{thm:lefschetz_is_safe}]
       During this proof, we assume that the Lefschetz cancellation of the pair $\bdb$ does not remove the row $\bth\bdb$ or the column $\dth\bdb$ from the matrix $\mD$. This assumption simplifies the calculations, and as after the cancellation, $\bth\bdb$ and $\dth\bdb$ correspond to a zero row and a zero column, respectively, and therefore do not affect the argument. We will prove only the part for the boundary matrix, as the proof for the coboundary matrix is analogous.
    Suppose the claim is not true, and let $\cy$ be the first column of $\hmUaft_n$ such that $\hmU_n[:,y] \neq \hmUaft_n[:,y]$.
    Since $\hmU_n[:,y]$ describes a~series of column additions, the statement that $\hmU_n[:,y] \neq \hmUaft_n[:,y]$ means that there is a~difference between columns added to $\hmB_n[:,y]$ and to $\hmBaft_n[:,y]$ during the lazy reduction. Let $x_1,x_2,x_3...x_n$ be a~series of indexes of columns added to $\hmB_n[:,y]$, and $\hat{x}_1,\hat{x}_2,\hat{x}_3...\hat{x}_k$, a~series of indexes of columns added to $\hmBaft_n[:,y]$, both arranged in the order of addition, that is, from the latest birth to earliest. 
    Our goal is show, that both series have to be identical, except fact that $\dth\bds$ may shown in $(x_i)_{i=0}^{n}$. Hence by contradiction assume that they are differ and $j$ is the greatest index such that for all $i \leq j$ we have that $x_i = \hat{x}_i$.

    \noindent Consider the state of the $\cy$-th column at the stage when the two series diverge. 
     In particular, denote 
        $\cp := \low \left(\sum\limits_{i = 1}^{j} \hmR_n[:,\cx_i{}] + \hmB_n[:,\cy] \right)$ 
        and $\cq := \low \left( \sum\limits_{i = 1}^{j}\hat{R}_n[:,\hat{x}_i] + \hmBaft_n[:,\cy] \right)$.

    \textbf{Case 1.}
    If $x_{j+1} = \dth\bds$, then $\hat{x}_{j+1}$ cannot be equal $\dth\bds$ as this column is zero column in $\hmBaft_n$. 
    However, this implies that $ p = \bth\bds$. 
    By Observation~\ref{obs:column-in-R-as-a-sum-of-columns-in-D} we can find a~set $A$ 
     such that
    $\sum\limits_{\cx  \in A} D_n[:,\cx]= \sum\limits_{i = 1}^{j} \hmR_n[:,\cx_i{}] + \hmB_n[:,\cy].$ 
    Therefore $p =\low \left(\sum\limits_{\cx  \in A} D_n[:,\cx] \right) = \bth\bds$, which implies that $\sum\limits_{\cx  \in A} D_n[\bth\bds,\cx] = 1$.


    However, since $x_i = \hat{x}_i$ for $i \leq j$, $q$ is described by the same set of indexes, i.e., $q = \low \left(\sum\limits_{\cx  \in A} \hat{\hmB}_n[:,\cx] \right)$.  Since $\hmBaft_n[:,x] = \hmB_n[:,x] + \hmB_n[\bth\bds,x] \cdot \hmB_n[:,\dth\bds]$, this yields that:

    \begin{align*} 
    \cq  &= 
    \low \left( \sum\limits_{x \in A}\hmB_n[:,x] + \hmB_n[:,\dth\bds]\sum\limits_{x \in A}\hmB_n[\bth\bds,x] \right)
    = \low \left( \sum\limits_{x \in A}\hmB_n[:,x] + \hmB_n[:,\dth\bds] \right)\\
    &= \low \left(\sum\limits_{i = 1}^{j} \hmR_n[:,\cx_i{}] + \hmB_n[:,\cy] + \hmB_n[:,\dth\bds] \right)\\ 
    &= \low \left(\sum\limits_{i = 1}^{j+1} \hmR_n[:,\cx_i{}] + \hmB_n[:,\cy] \right),
    \end{align*}
    where the last equality holds because the shallowness of $\bds$ implies that $\hmB_n[:,\dth\bds]=\hmR_n[:,\dth\bds]$.
    In other words, $\hat{x}_{j+1} = x_{j+2}$. 
    In fact, $k+1=n$ and $\hat{x}_{i} = x_{j+1}$ for all $i\in\{j+1,\ldots,k\}$.
        This will follow from an adaptation of \textbf{Case 2.}.
        In particular, since the series differ only by $\dth\bds$ it follows that columns $\hmU[:, y]$ and $\hmUaft[:, y]$ 
            also differ exactly by the entry corresponding to $\dth\bds$. To avoid confusion when analyzing \textbf{Case 2} under the assumption that \textbf{Case 1} holds,
we insert a zero column between $\hat{x}_{j}$ and $\hat{x}_{j+1}$,
shifting the indices so that $\hat{x}_{j+1} = 0$ and the original $\hat{x}_{j+1}$ becomes $\hat{x}_{j+2}$.

\noindent \textbf{Case 2.}
Suppose that $x_{j+1} \neq \hat{x}_{j+1}$ and $x_{j+1}$ is different from $\dth\bds$.
By Theorem \ref{thm:cancelation}, the pairings in $\hmR_n$ and $\hat{\hmR}_n$ remain the same,
    but $x_{j+1} \neq \hat{x}_{j+1}$ implies that $p \neq q$. 
    Observe that the cancellation of $\bds$ does not change any row which lies below $\bth\bds$.
    Thus, all additions caused by conflicts in rows below $\bth\bds$ must remain unchanged.
    This means $p$ and $q$ are rows above $\bth\bds$, because $p \neq \bth\bds$ and $q$ cannot be equal $\bth\bds$.
    Using the canonical representation to construct the~set of indexes $A$ such that $\cp = \low \left(\sum\limits_{\cx \in A} \hmB_n[:,\cx]\right)$. 
    As both series of columns additions are the same,
        $\cq = \low \left( \sum\limits_{\cx \in A}\hmBaft_n[:,x] \right)$. 
    Moreover, as:
    
    \[\hmBaft_n[:,x] = \hmB_n[:,x] + \hmB_n[\bth\bds,x] \cdot \hmB_n[:,\dth\bds],\] 
    
    \noindent we have $ \cq  = \low \left( \sum\limits_{x \in A}\hmB_n[:,x] + \hmB[:,\dth\bds]\sum\limits_{x \in A}\hmB_n[\bth\bds,x] \right)$. However, as $p$ is above $\bth\bds$, we have $  \sum\limits_{\cx \in A} \hmB_n[\bth\bds,\cx] = 0$, so $q = \low \left(\sum\limits_{\cx \in A} \hmB_n[:,\cx]\right) = p$, a~contradiction. 

\end{proof}

\InversingPathMorseComplex*

\noindent Before presenting the proof of the Theorem \ref{thm:cancelation_is_path_reversing} let us fix some notation. For a~path $\rho$ in $G_\cV$ we will denote its startpoint as $\pbeg{\rho}$ and its endpoint as $\pend{\rho}$. 
If $x, y \in \rho$, by $\rho[x, y]$ we define a~subpath $\xi$ of $\rho$ such that $\pbeg\xi = x$ and $\pend\xi = y$. 
If $\eta$ and $\zeta$ are paths such that $\pend\eta = \pbeg\zeta$, we define their \emph{composition} as the path $\eta \cdot \zeta := \eta \cup \zeta$. 
Finally, if $y, x \in X$, by $\pathhs(y, x)$ we denote the set of paths $\theta \in G_{\cV}$ such that $\pbeg\theta = y$ and $\pend\theta = x$. 

\begin{proof}[Proof of Theorem \ref{thm:cancelation_is_path_reversing}]
    By Definition \ref{def:morse_complex}, the boundary coefficients of both $\cM(\cV)$ and $\cM(\cV^{-\rho})$ are determined by the number of paths joining the critical cells. 
    Thus, we prove the theorem by comparing the set of paths between $x, y \in \cM(\cV) \setminus \{t, s\}$ before and after reversing $\rho$.
    Note that the formula for $\hat{\mB}(x,y)$ (Definition~\ref{def:lefschetz-cancellation}) is nontrivial only when $\dim y - \dim x=1$, $\dim y=\dim t$ and $\dim x =\dim s$. 
    Hence, fix $k := \dim s$.

    \noindent Let $C=\pathhs_{\cV}(y, x)$, $C':=\pathhs_{\cV^{-\rho}}(y, x)$ and let $E:=C\cap C'$, that is, the set of paths unaffected by the reversal of $\varrho$. 
    Notice that every $\xi \in C$ is a~$k$-path.
    We claim that $\xi \in E$ if and only if $\xi \cap \rho = \emptyset$.
    It is straightforward that if $\xi \cap \rho = \emptyset$ then $\xi\in E$.
    To see the other implication, assume the contrary, that is, that there exists $p \in \xi \cap \rho$.
    We necessarily have $p \neq s$. Indeed, if $p = s$, since $x \neq s$, there exist a~cell on $\xi$ after $p$. As $p = s$ is a~critical cell, such cell must be a~$(k-1)$-dimensional, contradicting the fact that $\xi$ is a~$k$-path. Similarly, if $p = t$, there exists a~cell on $\rho$ preceding $p$, as $t \neq y$. Such cell would have to be $(k+2)$-dimensional, contradicting the fact that $\xi$ is a~$k$-path. 
    Therefore, $p$ belongs to a vector $v \in \cV$, which also implies that $p \neq x, y$. Let $q$ be the second component of the vector $v$. We claim that $q \in \xi$. 
    If $\dim p = k$, then the successor of $p$ on $\rho$ must be $q$ and if $\dim p = k+1$, then $q$ must be the predecessor of $p$ on $\xi$. 
    However, since $v \subset \rho$, this vector no longer exists in $\cV'$. Thus $\xi \not\in C'$. 


    
    \noindent Consider another collection of paths, $A$, consisting of all paths $\alpha$ such that $\pbeg\alpha=y$ and $\alpha\cap\rho=\{\pend\alpha\}$, that is, the only element of $\alpha$ in $\rho$ is its endpoint.
    Analogously, let $B$ denote the collection of all paths $\beta$ starting on $\rho$, that is $\pend\beta=x$ and $\beta\cap\rho=\{\pbeg\beta\}$.
    We necessarily have $\dim\pend\alpha=\dim s$ for any $\alpha\in A$ and
        $\dim\pbeg\beta=\dim t$ for every $\beta\in B$.
    This implies that the set of endpoints of paths in $A$ and startpoints of paths in $B$ are disjoint.
    
    
    \noindent Define $r \colon A\rightarrow\NN$ as 
        $r(\alpha):=\#\{\beta\in B\mid \pend\alpha\pathh{\cV}\pbeg\beta\}$, that is, the number of paths in $B$ with the starting point further along $\rho$ than $\pend\alpha$.
    We claim that
    \begin{equation}\label{eq:counting_connecting_paths}
        \# C = \# E + \sum_{\alpha\in A} r(\alpha).
    \end{equation}
    Indeed, we can associate every choice of $\alpha\in A$ and $\beta\in B$ such that 
    $\pend\alpha\pathh{\cV}\pbeg\beta$ 
    with the unique path $\xi$ from $y$ to $x$
        constructed as a~concatenation of $\alpha$,
        the segment of $\rho$ from $\pend\alpha$ to $\pbeg\beta$,
        and $\beta$ 
        .
    \noindent Since $\xi \cap \rho \neq \emptyset$, we have $\xi\not\in E$.
    This shows ``$\geq$'' for \eqref{eq:counting_connecting_paths}. 
    To see the other inequality, it is is enough to observe that every $\xi\in C$ is either in $E$ or decomposes into
    $\alpha\in A$, $\beta\in B$ and the segment $\rho[\pend\alpha, \pbeg\beta]$, that is, $\xi$ is counted in $r(\alpha)$ in \eqref{eq:counting_connecting_paths}.
    
    \noindent Define analogously 
    $r'\colon A\rightarrow\NN$ as 
        $r'(\alpha):=\#\{\beta\in B\mid \pend\alpha\pathh{\cV'}\pbeg\beta\}$.
    With a~similar argument as before, we get
    \begin{equation}\label{eq:counting_connecting_revpaths}
        \# C' = \# E + \sum_{\alpha\in A} r'(\alpha).
    \end{equation}
    
    \noindent Since for any $\alpha\in A$ and $\beta\in B$ we have $\pend\alpha \neq \pbeg\beta$ and
        either $\pend\alpha \pathh{\cV} \pbeg\beta$ or $\pend\alpha \pathh{\cV'} \pbeg\beta$, 
        it follows that $r(\alpha) + r'(\alpha) = \# B$.
    Hence, we have
    \begin{align*}
        \sum_{\alpha\in A} r(\alpha) + \sum_{\alpha\in A} r'(\alpha) =
            \sum_{\alpha \in A} \# B = \# A \cdot \# B.
    \end{align*}
    This implies that
    \begin{align*}
        \label{eq:paths_after_inv}
        \# C + \# C' = 2 \cdot \#E + \#A \cdot \# B.
    \end{align*}
    
    \noindent The last thing to observe is that every path $\alpha\in A$ can be extended to a~path from $y$ to $s$.
    Even more, these extensions of $\alpha$'s generate all paths from $y$ to $s$.
    Thus, we have $\# A=\#\pathhs_\cV(y,s)$. Analogously, $\# B =\#\pathhs_\cV(t,x)$.
    Since $2 \cdot \# E = 0 \pmod 2$, the above equation translates into the following congruence
    \begin{equation*}
        \# \pathhs_{\cV^{-\rho}}(y,x)=
            \# \pathhs_{\cV}(y,x) + \# \pathhs_{\cV}(y,s) \cdot \# \pathhs_{\cV}(t,x) \pmod 2, 
    \end{equation*}
    which directly translates into the formula for $\hat{\mB}(x,y)$ (Definition~\ref{def:lefschetz-cancellation}).
    This concludes the proof.
\end{proof}

\WeCanOnlyLose*

\begin{proof}

      Assume that we perform a~transposition with a~component of $\bdt$. Because of Theorem~\ref{thm:mixed_transposition}, we consider only birth--birth and death--death transpositions.

      \noindent Consider birth--birth transposition. Because the incoming relation arrows of $\bds$ depend only on $\hmU[:,\dth\bds]$, it suffices to check how transposition affects the column $\hmU[:,\dth\bds]$. If $\dth\bds < \dth\bdt$, then the update is adding row $\dth\bdt$ to $\dth\bds$. However, since $\hmU_{n}$ is upper-triangular, $\hmU_{n}[\dth\bdt,\dth\bds] = 0$, so we do not modify $\hmU_{n}[:,\dth\bds]$.  If $\dth\bdt < \dth\bds$, as we do not switch pairing, we can assume that we are in case (1) of Theorem~\ref{thm:birth_transposition}. However, because $\dth\bdt < \dth\bds$, $\mD_{n}^{\bds,\bdt}[\bth\bdt,\dth\bds] = 1$ implies, by Observation~\ref{obs:quotient_relation}, that $\bdt \homo \bds$, so the update pattern yields:   $\hmUaft_{n}[\dth\bdt,\dth\bds] = \hmUbef_{n}[\dth\bdt,\dth\bds]+\hmUbef_{n}[\dth\bds,\dth\bds] = 1 + 1 = 0.$ The proof for the death--death transposition case is similar. 
\end{proof}

\DmfHomotopy*

\begin{proof}[Proof o Theorem \ref{thm:dmf-homotopy}]
    Let $t \in (0,1)$. We first prove that for every $v_0 \in \cV_{f_0}$, there exists $v_t \in \cV_{f_t}$ such that $v_0 \subseteq v_t$. Fix $v_0 \in \cV_{f_0}$ and $a, b \in v_0$.
    If $v_0 \in \crit(\cV_{f_0})$, then $a$ is equal $b$, otherwise one is facet of another. In both cases $f_0(a)=f_0(b)$ and $f_1(a)=f_1(b)$, since $\cV_{f_0}=\cV_{f_1}$. Hence, $f_t(a)=f_t(b)$, which implies that $a, b \in v_t$ for some $v_t \in \cV_{f_t}$. Notice that such $v_t$ is unique, since $\cV_{f_t}$ is a~partition.
    This defines a~function $\Lambda \colon \cV_{f_0} \to \cV_{f_t}$, assigning $v_0$ to $v_t$. Next, we prove that $\Lambda$ is injective. Let $v_0, v'_0 \in \cV_{f_0}$ be distinct, and let $a~\in v_0$ and $a' \in v'_0$. We assume wlog that $f_0(a) < f_0(a')$. If $f_1(a) < f_1(a')$, then $f_t(a) < f_t(a')$, thus $v_t \neq v'_t$. Otherwise $f_0(a) < f_0(a')$ and $f_1(a) > f_1(a')$. This implies that $a'$ and $a$ are incomparable in the face relation. Note that this implies that any $b \in v_0$ is incomparable with any $b' \in v'_0$. Since elements of $\cV_{f_t}$ are connected, it follows that $v_t \neq v'_t$. Finally, since $v_0 \subseteq \Lambda(v_0)$,
    \[
        \card{X}=\card{\bigcup \cV_{f_0}}=
        \sum_{v_0 \in \cV_{f_0}} \card{v_0}\leq
        \sum_{v_0 \in \cV_{f_0}} \card{\Lambda(v_0)}\leq
        \sum_{v_t \in \cV_{f_t}} \card{v_t}=
        \card{\bigcup \cV_{f_t}}=\card{X}.
    \]
    It follows that $\sum_{v_0 \in \cV_{f_0}} \card{v_0} = \sum_{v_0 \in \cV_{f_0}} \card{\Lambda(v_0)}$. This is possible only if $\card{v_0} = \card{\Lambda(v_0)}$ for $v_0 \in \cV_{f_0}$ and, consequently, $v_0=\Lambda(v_0)$.    
\end{proof}

\MoveRightProp*

\begin{proof}[Proof of Proposition \ref{prop:pre_allowed_move_right}]
    Directly from the construction we get that for every $x \in X$:  $\mf(x) \leq \omf(x)$ and whenever $\mf(\cx)\neq \omf(\cx)$, then $ \mf(x) < \delta < \omf(x) < \xi$. Finally, $\omf(\cx)$ belongs to the interval $[\delta,\xi]$ exactly when $\omf(\cx)\ne \mf(\cx)$, and $\omf(\cx)\notin[\delta,\xi]$ exactly when $\omf(\cx)=\mf(\cx)$.

 \noindent If $\mf(x)=\mf(y)$, then $(\cx,\cy)$ is a~vector and $\cx \pathh{\cV_{\mf}} a$ iff $y \pathh{\cV_{\mf}} \ca$. This guarantees that the value of $\cx$ and $\cy$ is modified in the same way, which yields that $\omf(\cx) = \omf(\cy)$. From the other side if $\omf(\cx) = \omf(\cy)$, then we have two cases. If $\omf(\cx) \in [\delta,\xi]$, then both values were modified in the same way, and their equality is equivalent to the fact that $t_x = t_y$. Hence, $\mf(\cx) = \mf(\cy)$. Otherwise if $\omf(x) \not\in [\delta,\xi]$ then $\mf(\cx) = \omf(\cx) = \omf(\cy) = \mf(\cy)$.

\noindent So we get that $\mf(\cx)=\mf(\cy)$ iff $\omf(\cx) = \omf(\cy)$. Now let us prove that $\omf$ is a~$\dmf$. If $\hmB(\cx,\cy)=1$, then consider three cases. (1) If both values of $\cx$ and $\cy$ were either modified, or not modified, then inequality holds. (2) If value of $\cx$ was not modified, while the value of $\cy$ was modified, then we get $\omf(\cx) \leq \omf(\cy)$ from the fact that $\mf(\cy) \leq \omf(\cy)$. (3) If the value of $x$ was modified while $y$ was not, then $\hmB(\cx,\cy) = 1$ implies $y \pathh{} x \pathh{} \bth\bds$. Because $\cy$ was not modified, then from the assumption $\xi < \mf(\cy)$, and because $\omf(x) < \xi$, we get the desired inequality.

\noindent Pairing property and almost injectivity follows directly from the fact that $\mf(x) = \mf(y)$ iff $\omf(x) = \omf(y)$. So $\omf$ is indeed a~$\dmf$ and $\cV_{\mf} = \cV_{\omf}$. Also from the construction we see that $\omf$ is equal $\mf$ on critical cells, except $\bth\bds$. So $\omf$ is pre-allowed move.
\end{proof}

\MoveDownProp*

\begin{proof}[Proof of Proposition \ref{prop:pre_allowed_move_down}]
    Proof is dual to proof of Proposition \ref{prop:pre_allowed_move_right}
\end{proof}

\RevPathProp*

\begin{proof}
    We only need to prove that $\omf$ is a~$\dmf$; the rest is immediate.
    The pairing property and the almost injective property follow directly from the construction, as does the fact that $\mf^{-1}([\mf(\bth\bds),\mf(\dth\bds)]) = \rho$. To show weak monotonicity, take $\cx,\cy \in X$ such that $D(\cx,\cy) = 1$ and consider three cases. (1) If both cells are on the path or outside the path, then weak monotonicity follows from  the construction. (2) If $\cx$ is on the path, and $\cy$ is not, then $\omf(y) = \mf(\cy) > \mf(\dth\bds)$ from the assumption and as $\omf(x) \in [\mf(\bth\bds),\mf(\dth\bds)]$, then we get the desired inequality. (3) Similarly, if $\cy$ is on the path and $\cx$ is not, then $\mf(\cy), \omf(y) \in [\mf(\bth\bds),\mf(\dth\bds)]$ and $\omf(x) = \mf(x) < \mf(\bth\bds)$
\end{proof}

\section{Complexity}\label{sec:complexity}

\subsection{Optimization}
 
 Before we calculate total complexity of the algorithm, let us show that one can check condition $\mD_{n+1}^{\bds,\bdt}[\bth\bdt,\dth\bds] = 1$ from Theorem $\ref{thm:birth_transposition}$ in linear time. In this subsection, we assume that $\bds,\bdt$ are $n$-dimensional birth-death pairs and whenever we refer to matrices $\mD,\mR,\mU$ the reader should assume that we mean the $n+1$ dimensional matrices. 
 Let $\mV = \mU^{-1}$, so $\mD\mV = \mR$. Dually matrices $\cmD,\cmR,\cmV,\cmU$ are $n$-dimensional dual counterparts of $\mD,\mR,\mU,\mV$.

\begin{observation}\label{obs:lef_matrix_and_trans}
    If $\mf(\dth\beta)<\mf(\dth\alpha)$, $\mf(\bth\alpha) < \mf(\bth\beta)$ and $\mD^{\bds,\bdt}[\bth\bdt,\dth\bds] = 1$ then $\bth\beta \cohomo \bth\alpha$.
\end{observation}

\noindent Because column $\bth\bdt$ will be added to column $\bth\bds$ during the lazy reduction of $\mD^{\perp,\bds,\bdt}$, the same holds for $D^{\perp}$, from Theorem \ref{thm:lefschetz_is_safe}. This guaranties that the only moment when we need to calculate $\mD^{\bds,\bdt}[\bth\bdt,\dth\bds] = 1$ is when $\bdt$ is in the bottom-left quadrant of $\bds$ and passes by $\bds$ when moving to the right.

\begin{proposition}\label{prop:optimization}
    If $\mf(\dth\bdt) < \mf(\dth\bds)$, $\mf(\bth\bdt) < \mf(\bth\bds)$ and $\bth\bdt$ and $\bth\bds$ are consecutive in the $\mf$-order then $\mD^{\bds,\bdt}[\bth\bdt,\dth\bds] = \mR[\bth\bdt,\dth\bds]$.
\end{proposition}

\begin{proof}
   During this proof, we adopt the same convention regarding Lefschetz cancellation as in the proof of Theorem \ref{thm:lefschetz_is_safe}. By $\mD_{\bdb}$ we denote the matrix $\mD$ after Lefschetz cancellation of a birth-death pair $\bdb$. 
   Observe that it follows from the lazy reduction that $\mR[:,\dth\bds] = \sum\limits_{\mV[\dth\bda,\dth\bds] = 1} \mD[:,\dth\bda]$. 
    Since $\mV$ is positive on the diagonal, we can define $A := \{ \bda \mid \mV[\dth\bda,\dth\bds] = 1 \} \setminus \{\bds\}$. 
    Note that all pairs in $A$ lie in the bottom-right quadrant of $\bds$. Then 
    \begin{align*}
        \mR[:,\dth\alpha] = \mD[:,\dth\bds]+\sum\limits_{\bda \in A} \mD[:,\dth\bda].
    \end{align*}

    \noindent Hence we have $\mD[:,\dth\alpha] = \sum\limits_{\bda \in A} \mD[:,\dth\bda] + \mR[:,\dth\bds]$. 
    Let $\bdb$ be an arbitrary shallow pair of $\mD$ lying in the bottom-right quadrant of $\bds$.
    Then, after the Lefschetz cancellation of $\bdb$ we get:

\begin{align*}
    \mD_{\bdb}[:,\dth\bds] = \mD[:,\dth\bds] +  \mD[\bth\bdb,\dth\bds]\cdot \mD[:,\dth\bdb] \\
    = \left(\mR[:,\dth\bds] + \sum\limits_{\bda \in A} \mD[:,\dth\bda]\right) + \left(\mR[\bth\bdb,\dth\bds] +\sum\limits_{\bda \in A} \mD[\bth\bdb,\bda] \right) \cdot \mD[:,\dth\bdb] \\
     = \left(\mR[:,\dth\bds] + \sum\limits_{\bda \in A} \mD[:,\dth\bda]\right) + \mR[\bth\bdb,\dth\bds] \cdot \mD[:,\dth\bdb] +\sum\limits_{\bda \in A} \mD[\bth\bdb,\bda] \cdot \mD[:,\dth\bdb] \\
          = \left(\mR[:,\dth\bds] + \sum\limits_{\bda \in A} \mD[:,\dth\bda]\right) + 0 +\sum\limits_{\bda \in A} \mD[\bth\bdb,\bda] \cdot \mD[:,\dth\bdb] \\
        = \mR[:,\dth\bds] + \sum\limits_{\bda \in A}\left( \mD[:,\dth\bda]+\mD[\bth\bdb,\bda] \cdot \mD[:,\dth\bdb]\right)  \\
        = \mR[:,\dth\bds] + \sum\limits_{\bda \in A} \mD_{\bdb}[:,\dth\bda]
\end{align*}
where $\mR[\bth\bdb,\dth\bds] = 0$ is a consequence of $\bdb$ lying in the bottom-right quadrant of $\bds$ in the persistence diagram. Now, observe, that whenever a canceled pair $\bdb$ is in $A$, we get: 
\begin{align*}
     \mR[:,\dth\bds] + \sum\limits_{\bda \in A} \hat{\mD}_{\bdb}[:,\dth\bda] 
        &= \mR[:,\dth\bds] + \hat{\mD}_{\bdb}[:,\dth\bdb]+\sum\limits_{\bda \in A\setminus\{\bdb\}} \hat{\mD}_{\bdb}[:,\dth\bda] \\
        &= \mR[:,\dth\bds] +\sum\limits_{\bda \in A\setminus\{\bdb\}} \hat{\mD}_{\bdb}[:,\dth\bda],
\end{align*}
The term $\hat{\mD}_{\bdb}[:,\dth\bdb]$ vanishes because:
\begin{align*}
    \hat{\mD}_{\bdb}[:,\dth\bdb] = \hat{\mD}[:,\dth\bdb] + \hat{\mD}[\bth\bdb,\dth\bdb] \cdot \hat{\mD}[:,\dth\bdb] = \hat{\mD}[:,\dth\bdb] + \hat{\mD}[:,\dth\bdb] = 0,
\end{align*}
where  $\hat{\mD}[\bth\bdb,\dth\bdb]  = 1$ follows from the shallowness of $\bdb$ in $\hat{\mD}$.
Hence, canceling a shallow pair from the bottom-right quadrant of $\bds$ reduces the size of $A$ by one (if it is a pair from $A$) or keeps it unchanged. 

\noindent We can apply the above reasoning to cancel pairs from the bottom-right quadrant of $\bds$ iteratively.
Denote by $\mD^{\bds}$ the boundary matrix after erasing all the pairs in the bottom-right quadrant of $\bds$. 
In particular, every pair in $A$ is canceled. 
Therefore, eventually, we have $\mD^{\bds}[:,\dth\bds] = \mR[:,\dth\bds]$.


\noindent Finally, observe that since $\bth\bdt$ and $\bth\alpha$ are consecutive in the $\mf$-order all pairs that lie in the bottom-right quadrant of $\bdt$ form a subset of the set of pairs that lies in the bottom-right quadrant of $\bds$. 
This gives us that $\mD^{\bds,\bdt}[:,\dth\bds] = \mD^{\bds}[:,\dth\bds] = \mR[:,\dth\bds].$

\end{proof}

\noindent Of course, the fact that it is enough to check an entry in \(\mR\) does not solve all difficulties.
During the journey to the diagonal we update only the matrix \(\mU\), 
    but we still need an efficient approach to quickly compute \(\mR[\bth\bdt,\dth\bds]\) without performing the lazy reduction on the updated boundary matrix.
Thus, we prove that the previously known update pattern for \(\mU\) induces an update pattern for \(\mV\).

\begin{proposition}\label{prop:update_V}
 Assume that $\bth\bds,\bth\bdt$ are consecutive rows in $\mD$  and $\mf(\dth\bdt) < \mf(\dth\bds)$. Then the transposition between $\bth\bdt$ and $\bth\bds$ causes an update of $V$ only if $\mU$ demands the update. The update pattern is given by formula:

$$\hat{\mV}[:,\dth\bds] = \mV[:,\dth\bds] - \mV[:,\dth\bdt]$$
\end{proposition}

\begin{proof}
    Let $\mU$ denote the matrix of homological relations before transposition and $\hat{\mU}$---the matrix of homological relations after transposition. 
    Similarly, denote $\mV = \mU^{-1}$ and $\hat{\mV}=\hat{\mU}^{-1}$. 
    Clearly, if $\mU = \hat{\mU}$ then $\mV = \hat{\mV}$. 
    By Theorem \ref{thm:birth_transposition} whenever $\mU \neq \hat{\mU}$ then $\hmUaft[\dth\bdt,:] = \hmUbef[\dth\bdt,:]+\hmUbef[\dth\bds,:]$. 
    However, this means that $\hmUaft = A_{\bdt,\bds} \mU $ where $A_{\bdt,\bds}$ is an elementary matrix which adds row $\bds$ to $\bdt$. Then $\hmUaft^{-1} = (A_{\bdt,\bds}\mU )^{-1} = \mU^{-1}A_{\bdt,\bds}^{-1} =  \mV A_{\bdt,\bds}^{-1}$. 
    Note that the inverse of the elementary matrix $A_{\bdt,\bds}$ is another elementary matrix. 
    In particular, when multiplied from left, the result is subtraction of column $\bdt$ from column $\bds$.
    Since we work over $\ZZ_2$, the subtraction is equivalent to an addition.
\end{proof}

\noindent Therefore, after each move we are able to update (in linear time!) not only the matrix $\mU$ but also $\mV$. This gives us the opportunity to utilize the following theorem.

\begin{proposition}\label{prop:calculating_R}
    The value of the updated entry $\mR[\bth\bdt,\dth\bds]$ can be calculated in linear time if matrix $\mV$ is known.
\end{proposition}

\begin{proof}
 As $\mR[:,\dth\bds] = \sum\limits_{\mV[x,\dth\bds] = 1} \mD[:,x]$ then $\mR[\bth\bdt,\dth\bds] = \sum\limits_{\mV[x,\dth\bds] = 1} \mD[\bth\bdt,x]$. 
 Determining whether the entry of $\mV$ non-empty takes $\mathcal{O}(1)$ time, 
    as it is enough to iterate through the row $\mD[\bth\bdt,:]$.
\end{proof}

\noindent This indeed gives us technique to decide about update of $\mU$ and $\mV$ from Theorem \ref{thm:birth_transposition} in linear time. Moreover, we can formulate dual propositions for the Theorem \ref{thm:death_transposition}.

\begin{proposition}\label{prop:optimization2}
The decision whether the update given by Theorem \ref{thm:death_transposition} is required can be resolved in linear time assuming information from the matrix $\mV^{\perp}$.
\end{proposition}

\begin{proof}
    Follows dually to proofs of Propositions \ref{prop:optimization}, \ref{prop:update_V}, \ref{prop:calculating_R}.
\end{proof}

\noindent Now we are ready to estimate the complexity of the algorithm.

\subsection{Estimation}

\noindent Denote by $n,c$ respectively the number of cells in $\cV^{k}_{\mf}$ and the number of pairs in $\bd_{k}(h)$. 
  Assume that $\cV^{k}_{h}$ is given by graph $G_k\coloneqq(X_{k},E_{k})$, that is a restriction of graph $G_\cV$. 
  Moreover, for a fixed birth-death pair $\bds$, define $m(\bds)$ as the number of birth-death pairs in area $[\mf(\bth\bds),\mf(\dth\bds)] \times [-\infty,\infty] \cup [-\infty,\mf(\bth\bds)]\times[-\infty,\mf(\dth\bds)]$ of persistence diagram. Now note that:

\begin{bracketenumerate}

    \item Topological sort of the graph takes $\mathcal{O}(n + \card{E_{k}})$, and 
        the complexity of finding the number of paths between one pair is $\mathcal{O}(n+\card{E_k})$. 
        Therefore, finding a reversible pair takes  $\mathcal{O}(c \cdot (n+\card{E_k}))$ in the worst case. 

    \item Every pair $\bds$ has at most $m(\bds)$ points generating forbidden regions; 
        therefore, sorting these regions along corresponding axis costs $\mathcal{O}(m(\bds)\log m(\bds))$; once sorted, testing for intersections between the two types takes $\mathcal{O}(m(\bds))$ time. 

    \item For a fixed birth-death pair $\bds$ we perform at most  $m(\bds)$ allowed moves.

    \item Every allowed move is followed by computation of new filtration values which takes $\mathcal{O}(n)$; this follows by Propositions \ref{prop:pre_allowed_move_right} and \ref{prop:pre_allowed_move_down}.\label{complex:move}

    \item An allowed move requires checking the criterion for the update given by Theorems \ref{thm:birth_transposition} and \ref{thm:death_transposition};
        this takes at $\mathcal{O}(n)$
            (see Propositions \ref{prop:optimization} and \ref{prop:optimization2}).  \label{complex:criterion}
    \item \label{complex:update} Every update of $\hmU,\mV,\cmU,\cmV$ demands $\mathcal{O}(c)$ time as adding appropriate columns or rows.
    \item Due to \ref{complex:move},\ref{complex:criterion}
     and \ref{complex:update} the complexity of Step \ref{step:move} is $\mathcal{O}(m(\bds) \cdot n)$ 
    \item Reversing a path $\rho$, from Step \ref{step:reverse} takes $\mathcal{O}(\card{\rho})$ time
\end{bracketenumerate}

\noindent The final worst case scenario is $\mathcal{O}(c \cdot n)$ as $m(\bds)$ may be equal to $(c-1)$. However, in many cases, the observed complexity may be significantly smaller (when we need to pass over only a few pairs) and close to linear.


\end{appendices}

\end{document}